\documentclass[11pt,a4paper]{amsart}

\usepackage{subfig}
\usepackage{tikz-cd}
\usepackage{tikz}
\usepackage{caption}
\usepackage{amssymb}
\usepackage[all]{xy}
\usepackage{amsthm}
\usepackage{mathrsfs}
\usepackage{hyperref}
\usepackage{xcolor}
\usepackage{bm}
\definecolor{green}{RGB}{0,127,0}
\definecolor{blue}{RGB}{0,0,150}
\hypersetup{
    colorlinks=true,
    linkcolor=blue,
    filecolor=magenta,      
    urlcolor=cyan,
    bookmarks=true,
    citecolor=green
}

\usepackage{graphicx}

\usepackage{setspace}
\usepackage[top=2.6cm, bottom=2.6cm, left=2.6cm, right=2.6cm]{geometry}

\usepackage{graphicx}

\DeclareMathOperator{\N}{\mathbf{N}}
\DeclareMathOperator{\Z}{\mathbf{Z}}

\DeclareMathOperator{\Hom}{\textup{Hom}}
\DeclareMathOperator{\GL}{\textup{GL}}

\DeclareMathOperator{\Ext}{\textup{Ext}}
\DeclareMathOperator{\C}{\mathbf{C}}
\DeclareMathOperator{\Q}{\mathbf{Q}}

\DeclareMathOperator{\Ind}{\textup{Ind}}

\DeclareMathOperator{\A}{\mathbf{A}}
\DeclareMathOperator{\D}{\mathbb{D}}

\DeclareMathOperator{\im}{\textup{im}}
\DeclareMathOperator{\supp}{\textup{supp}}

\DeclareMathOperator{\End}{\textup{End}}

\DeclareMathOperator{\OO}{\mathcal{O}}
\DeclareMathOperator{\B}{\mathcal{B}}

\DeclareMathOperator{\id}{\textup{id}}

\DeclareMathOperator{\gl}{\mathfrak{gl}}

\DeclareMathOperator{\Ob}{\textup{Ob}}

\DeclareMathOperator{\Perv}{\textup{Perv}}

\DeclareMathOperator{\Irr}{\textup{Irr}}

\DeclareMathOperator{\ICC}{\mathcal{IC}}
\DeclareMathOperator{\rank}{\textup{rank}}

\DeclareMathOperator{\Coh}{\textup{Coh}}
\DeclareMathOperator{\Tor}{\textup{Tor}}
\DeclareMathOperator{\Res}{\textup{Res}}
\DeclareMathOperator{\G}{\mathbf{G}}
\DeclareMathOperator{\pt}{\textup{pt}}
\DeclareMathOperator{\SL}{\textup{SL}}
\DeclareMathOperator{\rk}{\textup{rk}}

\def\LL{\boldsymbol{\Lambda} \hspace{-.06in} \boldsymbol{\Lambda}}

\newtheorem{theorem}{Theorem}[section]
\newtheorem{lemma}[theorem]{Lemma}
\newtheorem{proposition}[theorem]{Proposition}
\newtheorem{cor}[theorem]{Corollary}

\theoremstyle{definition}

\theoremstyle{remark}
\newtheorem{remark}[theorem]{Remark}

\numberwithin{equation}{section}

\newcommand\Overline[2][1pt]{%
    \begin{tikzpicture}[baseline=(a.base)]
      \node[inner xsep=0pt,inner ysep=1.5pt] (a) {$#2$};
      \draw[line width= #1] (a.north west) -- (a.north east);
    \end{tikzpicture}
    }

\begin{document}

\title[Perverse sheaves on the stack of coherent sheaves]{Perverse sheaves with nilpotent singular support on the stack of coherent sheaves on an elliptic curve}

\author{Lucien Hennecart}

\address{Universit\'e Paris-Saclay, CNRS,  Laboratoire de math\'ematiques d'Orsay, 91405, Orsay, France}

\email{lucien.hennecart@universite-paris-saclay.fr}

\date{\today}

\begin{abstract}
We define a stratification of the moduli stack of coherent sheaves on an elliptic curve which allows us (1) to give an explicit description of the irreducible components of the global nilpotent cone of elliptic curves, (2) to establish an explicit bijection between the simple objects of the category of perverse sheaves defined by Schiffmann to categorify the elliptic Hall algebra (the so-called spherical Eisenstein sheaves) and the irreducible components of the global nilpotent cone and (3) to give an explicit description and parametrization of the perverse sheaves on the moduli stack of coherent sheaves on an elliptic curve having nilpotent singular support. Along the way, we find a combinatorial parametrization of the irreducible components of the semistable locus of the elliptic global nilpotent cone. The comparison with Bozec's parametrization leads to an interesting combinatorial problem.
\end{abstract}

\maketitle

\tableofcontents

\section{Introduction}
\label{intro}
In this paper, we study the relationship between the irreducible components of the global nilpotent cone of an elliptic curve and the simple objects of some categories of perverse sheaves on the stack of coherent sheaves (Eisenstein perverse sheaves) relevant in the geometric Langlands program. The global nilpotent cone is a closed substack of the stack of Higgs bundle whose geometry has been studied in depth and is also an essential object in the geometric Langlands program (\cite{MR962524,MR899400}). The stack of Higgs bundles is the cotangent stack to the stack of coherent sheaves and the global nilpotent cone is a Lagrangian substack (\cite{MR962524,MR1853354}). When the underlying curve is the affine line, the global nilpotent cone for torsion sheaves of length $d$ is the stack quotient given by pairs of $d\times d$ matrices, the second one being nilpotent, modulo the simultaneous conjugation by the general linear group $\GL_d$. This stack plays a important role in Springer theory and in the microlocal study of character sheaves on reductive Lie algebras (\cite{MR2124171}). Spherical Eisenstein sheaves form a category of perverse sheaves on the stack of coherent sheaves and are the perverse sheaves appearing as shifted direct summands of the induction of trivial local systems (\cite{MR2942792}). In this paper, we define \emph{twisted spherical Eisenstein sheaves}. They form a bigger category: we take the simple constituents of inductions of arbitrary local systems. They can be defined for an arbitrary smooth, projective curve (over a finite field and in the $\ell$-adic setting or over a complex curve, working with $\Q$ or $\C$ coefficients). Twisted spherical Eisenstein perverse sheaves supported on some connected component of the stack of torsion sheaves have an explicit description in terms of local systems associated to representations of the symmetric groups and local systems on the curve. This is one of our main results. The particular structure of the stack of coherent sheaves on an elliptic curve which rests upon the description by Atiyah of the category of such objects (\cite{MR131423}) allowed Schiffmann to describe in explicit terms the whole category of spherical Eisenstein sheaves (\cite{MR2942792}). For curves of genus bigger that two, the description remains mysterious. The interest in this category lies in the fact that it gives a geometric categorification of the elliptic Hall algebra defined in \cite{MR2922373}. The elliptic Hall algebra is a deformation of the Hopf algebra of diagonally symmetric polynomials
\[
\LL^+=\C[x_1^{\pm 1},\hdots,y_1,\hdots]^{\mathfrak{S}_{\infty}}. 
\]
The combinatorial study of such a ring is part of the theory of \emph{multisymmetric} functions, which attempts to generalize to an arbitrary number of sets of variables the classical theory of Macdonald of symmetric functions (\cite{MR3443860}). The elliptic Hall algebra has now appeared in a great diversity of problems: in the study of the $K$-theory of the Hilbert scheme of the affine plane (\cite{MR3018956}), in skein theory of tori (\cite{MR3626565}), in diagrammatic categorification (\cite{MR3874690}),...

Motivated by a putative Lagrangian construction of a specialization of the elliptic Hall algebra in the spirit of Lusztig's semicanonical basis of quantum groups (\cite{MR1758244}), by the fact that in the context of quivers, Lusztig sheaves (\cite{MR1088333}) are in canonical one-to-one correspondence with irreducible components of Lusztig nilpotent variety (\cite{MR1458969}) and by a possible geometric interpretation of elliptic Kostka polynomials, which we leave for future investigations, we were led to the study of the characteristic cycle map from the category of spherical Eisenstein sheaves to Lagrangian cycles in the stack of Higgs bundles. Our main results are the description of the irreducible components of the global nilpotent cone of an elliptic curve, the unitriangularity of the characteristic cycle map and the description of simple perverse sheaves on the stack of coherent sheaves on an elliptic curve having nilpotent singular support. As a corollary, we deduce that the characteristic cycle map induces a canonical bijection between isomorphism classes of simple spherical Eisenstein sheaves and irreducible components of the global nilpotent cone.

We would like to mention that Bozec has an other approach to the description of the irreducible components of the global nilpotent cone which works for any genus (\cite{2017bozec}). His approach uses the Jordan type of the Higgs field while our is very specific to elliptic curves. An interesting combinatorial problem is to relate these two parametrizations. We do this in small ranks, where piecewise linear structures arise.

\subsection{The main results}
\label{mainresults}
We give a brief overview of the main results of this paper. We refer to relevant sections for more details on the notation.
\subsubsection{The irreducible components of the elliptic global nilpotent cone}
\label{irrcompnilcone}
We consider a complex elliptic curve $X$. The category of coherent sheaves on $X$ is denoted by $\Coh(X)$, the set of isomorphism classes of coherent sheaves by $\lvert\Coh(X)\rvert$ and the moduli stack of objects of $\Coh(X)$ by $\mathfrak{Coh}(X)$. We adopt similar notations for subcategories of $\Coh(X)$ when this makes sense. The class of a coherent sheaf $\mathcal{F}$ on $X$ is the pair $\alpha=(r,d)$ of its rank and its degree, $r=\rank \mathcal{F}$, $d=\deg\mathcal{F}$. It is an element of the monoid $\Z^+=\{(r,d)\in\Z^2\mid r>0 \text{ or ($r=0$ and $d>0$)}\}$. A coherent sheaf $\mathcal{F}$ on $X$ is said \emph{semistable} provided that for any proper nonzero subsheaf $\mathcal{G}$, the inequality $\frac{\deg\mathcal{G}}{\rank\mathcal{G}}\leq\frac{\deg\mathcal{F}}{\rank\mathcal{F}}$ holds. The quantity $\frac{\deg\mathcal{F}}{\rank\mathcal{F}}$ is called the slope of $\mathcal{F}$ and is denoted by $\mu(\mathcal{F})$. If the inequalities are always strict, the sheaf is said stable. Let $\mathscr{P}$ stand for the set of all partitions. In Section \ref{sstable}, we will define a partition of the set of isomorphism classes of rank $r$, degree $d$ semistable coherent sheaves on $X$,
\[
\lvert\Coh^{ss}_{\alpha}(X)\rvert=\bigsqcup_{\xi\in(\N^{\mathscr{P}})_{\delta}}\lvert\Coh^{ss}_{\alpha,\xi}(X)\rvert
\]
where $\alpha=(r,d)$, $\delta=\gcd(r,d)$ and $(\N^{\mathscr{P}})_{\delta}$ denotes the set of functions $\mathscr{P}\rightarrow \N$ such that $\sum_{\lambda\in\mathscr{P}}\xi(\lambda)\lvert\lambda\rvert=\delta$ (in particular, $\xi$ has finite support). 

The Harder-Narasinhan filtration of a coherent sheaf $\mathcal{F}$ on $X$ is the unique filtration $0=\mathcal{F}_0\subset \mathcal{F}_1\subset\hdots\subset\mathcal{F}_s=\mathcal{F}$ such that any successive subquotient $\mathcal{F}_j/\mathcal{F}_{j-1}$ is semistable and the sequence $(\mu(\mathcal{F}_{j}/\mathcal{F}_{j-1}))_{1\leq j\leq s}$ is strictly decreasing. This sequence is called the Harder-Narasimhan type of $\mathcal{F}$. More generally, for any $\alpha\in\Z^+$, we will define a partition of the set of isomorphism classes of coherent sheaves of Harder-Narasimhan type (abbreviated \emph{HN-type}) $\bm{\alpha}=(\alpha_1,\hdots,\alpha_s)$, where $\alpha_i=(r_i,d_i)\in\Z^+$:
\[
 \lvert\Coh_{\bm{\alpha}}(X)\rvert=\bigsqcup_{\bm{\xi}\in(\N^{\mathscr{P}})_{\bm{\delta}}}\lvert\Coh_{\bm{\alpha},\bm{\xi}}(X)\rvert
\]
where $(\N^{\mathscr{P}})_{\bm{\delta}}=\{\bm{\xi}=(\xi_1,\hdots,\xi_s)\in(\N^{\mathscr{P}})^s \mid\xi_i\in(\N^{\mathscr{P}})_{\delta_i} \text{ for $1\leq i\leq s$}\}$, $\delta_i=\gcd(\alpha_i)$ and $\bm{\delta}=(\delta_1,\hdots,\delta_s)$. An element $\bm{\xi}$ is said \emph{regular} provided that for any $1\leq i\leq s$ and any partition $\lambda\in\mathscr{P}$, $\xi_i(\lambda)\neq 0$ only if the length of $\lambda$ is $1$. The datum of a regular $\bm{\xi}$ is equivalent to the datum of the $s$-tuple of partitions $\bm{\lambda}=(\lambda_{1},\hdots,\lambda_s)$ where $\lambda_i=(j^{\xi_i((j))},j\geq 1)$, and $(j)$ denotes the partition of $j$ of length one. In particular, $\bm{\lambda}\in\mathscr{P}_{\bm{\delta}}=\{(\lambda_1,\hdots,\lambda_s)\in\mathscr{P}^s\mid \lvert\lambda_i\rvert=\delta_i\}$. In this case, we denote $\xi_i=\xi_{\lambda_i}$ and $\bm{\xi}=\bm{\xi}_{\bm{\lambda}}$.

This partition induces a locally closed stratification of the moduli stack of coherent sheaves of HN-type $\bm{\alpha}$:
\[
 \mathfrak{Coh}_{\bm{\alpha}}(X)=\bigsqcup_{\bm{\xi}\in(\N^{\mathscr{P}})_{\bm{\delta}}}\mathfrak{Coh}_{\bm{\alpha},\bm{\xi}}(X)
\]
and a stratification of the moduli stack of coherent sheaves of class $\alpha=(r,d)\in\Z^+$ in the numerical Grothendieck group:
\[
 \mathfrak{Coh}_{\alpha}(X)=\bigsqcup_{\bm{\alpha}\in HN(\alpha)}\bigsqcup_{\bm{\xi}\in(\N^{\mathscr{P}})_{\bm{\delta}}}\mathfrak{Coh}_{\bm{\alpha},\bm{\xi}}(X)
\]
where the first union runs over HN-types of rank $r$ and degree $d$, $\bm{\delta}=(\gcd{\alpha_i})_{1\leq i\leq s}$ if $\bm{\alpha}=(\alpha_i)_{1\leq i\leq s}$, $(\N^{\mathscr{P}})_{\bm{\delta}}=\{\bm{\xi}=(\xi_i)_{1\leq i\leq s}\in(\N^{\mathscr{P}})^{s}\mid \xi_i\in (\N^{\mathscr{P}})_{\delta_i}\}$.
The moduli stack of Higgs sheaves of type $\alpha=(r,d)$ is denoted $\mathfrak{Higgs}_{\alpha}(X)$ (Section \ref{globnil}). It contains the closed substack of nilpotent Higgs sheaves $\mathscr{N}_{\alpha}$ and there is a natural projection $\pi_{\alpha}:\mathfrak{Higgs}_{\alpha}(X)\rightarrow\mathfrak{Coh}_{\alpha}(X)$ forgetting the Higgs field. The restriction of this projection to the global nilpotent cone is denoted
\[
 \pi_{\alpha,\mathscr{N}}:\mathscr{N}_{\alpha}\rightarrow \mathfrak{Coh}_{\alpha}.
\]
For $\bm{\alpha}\in HN(\alpha)$, we let $\mathscr{N}_{\bm{\alpha}}$ be the union of the irreducible components of dimension $\dim\mathscr{N}$ of $\pi_{\alpha,\mathscr{N}}^{-1}(\mathfrak{Coh}_{\bm{\alpha}}(X))$. There is an induced projection
\[
 \pi_{\bm{\alpha},\mathscr{N}} : \mathscr{N}_{\bm{\alpha}}\rightarrow \mathfrak{Coh}_{\bm{\alpha}}(X)
\]
If furthermore $\bm{\xi}\in(\N^{\mathscr{P}})_{\bm{\delta}}$, we let $\mathscr{N}_{\bm{\alpha},\bm{\xi}}=\pi_{\bm{\alpha},\mathscr{N}}^{-1}(\mathfrak{Coh}_{\bm{\alpha},\bm{\xi}}(X))$, $\overline{\mathscr{N}_{\bm{\alpha},\bm{\xi}}}$ denote the closure of $\mathscr{N}_{\bm{\alpha},\bm{\xi}}$ in $\mathscr{N}_{\bm{\alpha}}$ and $\Overline[1.5pt]{\mathscr{N}_{\bm{\alpha},\bm{\xi}}}$ denote the closure of $\mathscr{N}_{\bm{\alpha},\bm{\xi}}$ in $\mathscr{N}_{\alpha}$. If $\bm{\xi}=\bm{\xi}_{\bm{\lambda}}$ for some uplet of partitions $\bm{\lambda}\in\mathscr{P}_{\bm{\delta}}$, that is $\bm{\xi}$ is regular, we let $\mathscr{N}_{\bm{\alpha},\bm{\lambda}}=\mathscr{N}_{\bm{\alpha},\bm{\xi}_{\bm{\lambda}}}$.

The following theorem gives a description of the irreducible components of $\mathscr{N}_{\alpha}$ and $\mathscr{N}_{\bm{\alpha}}$ different from that of \cite{2017bozec} and more convenient for the computation of the characteristic cycle of spherical Eisenstein sheaves (Theorem \ref{mainbij}).
\begin{theorem}
\label{mainirr}
 Let $\alpha\in\Z^+$. The irreducible components of $\mathscr{N}_{\alpha}$ are the $\Overline[1.5]{\mathscr{N}_{\bm{\alpha},\bm{\lambda}}}$ for $\bm{\alpha}\in HN(\alpha)$ and $\bm{\lambda}\in \mathscr{P}_{\bm{\delta}}$.
 
 Let $\bm{\alpha}\in HN(\alpha)$. The irreducible components of $\mathscr{N}_{\bm{\alpha}}$ are the $\overline{\mathscr{N}_{\bm{\alpha},\bm{\lambda}}}$ for $\bm{\lambda}\in \mathscr{P}_{\bm{\delta}}$.
\end{theorem}
We order the set of irreducible components of $\mathscr{N}_{\alpha}$ as follows:
\[
 \Overline[1.5]{\mathscr{N}_{\bm{\alpha},\bm{\lambda}}}\leq \Overline[1.5]{\mathscr{N}_{\bm{\beta},\bm{\nu}}}\iff\left\{\begin{aligned}
             &\mathfrak{Coh}_{\bm{\beta}}(X)\subset \overline{\mathfrak{Coh}_{\bm{\alpha}}}(X)\text{ is a strict inclusion}\\
&\text{or}\\
&\bm{\beta}= \bm{\alpha} \text{ and for any $1\leq i\leq s$, }\nu_i\leq \lambda_i                                                                                                                                                   \end{aligned}
\right.
\]
The first condition on the containment of Harder-Narasimhan strata can be reformulated in purely combinatorial terms using Harder-Narasimhan polytopes (see \cite[\S1.1 e)]{MR2942792}). We order similarly the set of irreducible components $\overline{\mathscr{N}_{\bm{\alpha},\bm{\lambda}}}$ of $\mathscr{N}_{\bm{\alpha}}$ (it reduces to the anti-dominant order on the uplet of partitions $\bm{\lambda}$).

\subsubsection{Bijectivity of the characteristic cycle map}
\label{bijectivityccmap}
Let $\alpha\in\Z^+$. We complete the space $\Z[\Irr(\mathscr{N}_{\alpha})]$ of functions $\Irr(\mathscr{N}_{\alpha})\rightarrow\Z$ with respect to the dimension of the support and let $\widehat{\Z[\Irr(\mathscr{N}_{\alpha})]}$ be the completed space. More precisely, $\widehat{\Z[\Irr(\mathscr{N}_{\alpha})]}$ consists of all sums $\sum_{i\geq 0}a_i[\Lambda_i]$ where $a_i\in\Z$ and $\Lambda_i$ is an irreducible component of $\mathscr{N}_{\alpha}$. The category of spherical Eisenstein perverse sheaves on $\mathfrak{Coh}_{\alpha}(X)$ is denoted by $\mathcal{P}^{\alpha}$ (Section \ref{eisensteinsheaves}). Isomorphism classes of simple spherical Eisenstein perverse sheaves on $\mathfrak{Coh}_{\alpha}(X)$ are parametrized by pairs $(\bm{\alpha},\bm{\lambda})$ where $\bm{\alpha}=(\alpha_1,\hdots,\alpha_s)\in HN(\alpha)$ and $\bm{\lambda}=(\lambda_1,\hdots,\lambda_s)$ is a $s$-tuple of partitions such that $\lvert\lambda_i\rvert=\delta_i:=\gcd{\alpha_i}$ (Theorem \ref{schiffmanneisenstein}). We let $\mathscr{F}_{\bm{\alpha},\bm{\lambda}}$ be the corresponding simple perverse sheaf. Isomorphism classes of simple perverse sheaves in $\mathcal{P}^{\alpha}$ are ordered in the same way as irreducible components of $\mathscr{N}_{\alpha}$:
\[
 [\mathscr{F}_{\bm{\alpha},\bm{\lambda}}]\leq [\mathscr{F}_{\bm{\beta},\bm{\nu}}]\iff\left\{\begin{aligned}
&\mathfrak{Coh}_{\bm{\beta}}(X)\subset \overline{\mathfrak{Coh}_{\bm{\alpha}}}(X)\text{ is a strict inclusion}\\
&\text{or}\\
&\bm{\beta}=\bm{\alpha} \text{ and for any $1\leq i\leq s$, }\nu_i\leq \lambda_i \end{aligned}\right.
\]

We complete the Grothendieck group of $\mathcal{P}^{\alpha}$ in a similar manner, in terms of the dimension of the support. The completed Grothendieck group is then denoted $\widehat{K_0(\mathcal{P}^{\alpha})}$. This completion was defined in \cite{MR2922373,MR2942792} by defining an adic valuation on $K_0(\mathcal{P}^{\alpha})$.
\begin{theorem}
\label{mainbij}
 The characteristic cycle map
 \[
  CC:\widehat{K_0(\mathcal{P}^{\alpha})}\rightarrow \widehat{\Z[\Irr(\mathscr{N}_{\alpha})]}
 \]
is an isomorphism of $\Z$-modules. This isomorphism is lower unitriangular with respect to the basis of simple perverse sheaves on the left and the basis of irreducible components of $\mathscr{N}_{\alpha}$ on the right, when ordered as above.
\end{theorem}
The isomorphism of Theorem \ref{mainbij} induces a canonical bijection between the set of isomorphism classes of simple sheaves in $\mathcal{P}^{\alpha}$ and $\Irr(\mathscr{N}_{\alpha})$. This bijection is described by $\mathscr{F}_{\bm{\alpha},\bm{\lambda}}\leftrightarrow \Overline[1.5]{\mathscr{N}_{\bm{\alpha},\bm{\lambda}}}$.

\subsubsection{Perverse sheaves with nilpotent singular support on the stack of coherent sheaves on an elliptic curve}
In our last main result, which we prove in Section \ref{psnilsing}, we describe explicitly the simple perverse sheaves on the stack $\mathfrak{Coh}_{\alpha}(X)$, $\alpha\in\Z^+$, whose singular supports are nilpotent (that is, a union of some of the irreducible components of the global nilpotent cone $\mathscr{N}_{\alpha}$). We introduce some piece of notation. Let $X$ be an elliptic curve. If $\alpha=(r,d)\in\Z^+$ is coprime, then the semistable locus $\mathfrak{Coh}_{(\alpha)}(X)$ is isomorphic to the stack quotient $X/\G_m$, where the action of $\G_m$ is trivial (Section \ref{sstable}). Therefore, we have a family of local systems $\mathscr{L}_{z}$ on $\mathfrak{Coh}_{(\alpha)}(X)$ indexed by characters of $\pi_1(X)\simeq \Z^2$ if we work with the topological fundamental group, or $\pi_1^{\acute{e}t}(X)\simeq \widehat{\Z^2}$ when considering the \'etale fundamental group. We prove in Proposition \ref{extensionlocsys} that such a local system extends to the whole of $\mathfrak{Coh}_{\alpha}(X)$. Theorem \ref{miclocchar} can be formulated as follows. The second statement is a consequence of Lemma \ref{linebun}.

\begin{theorem}
 \label{microlocchar}
The perverse sheaves on $\mathfrak{Coh}_{\alpha}(X)$ having a nilpotent singular support are precisely the perverse sheaves which are the simple constituents of the inductions of perverse sheaves of the form $\ICC(\mathscr{L}_z)$ on various $\mathfrak{Coh}_{\alpha}(X)$, $\alpha\in\Z^+$ coprime as defined above. Moreover, it suffices to take $\alpha=(0,1)$ and $\alpha=(1,d)$ for $d$ in a subset of the integers not bounded below.
\end{theorem}
Such perverse sheaves are called in this paper \emph{twisted spherical Eisenstein perverse sheaves} as when all the local systems $\mathscr{L}_z$ are trivial (that is $z=1$), then we obtain the perverse sheaves considered by Schiffmann in \cite{MR2942792} and called there \emph{spherical Eisenstein perverse sheaves}.

The proofs of Theorems \ref{mainirr} and \ref{mainbij} will be given in Section \ref{Proofs}. The proof of Theorem \ref{microlocchar} will be given in Section \ref{psnilsing}.

\subsection{Strategy of proof}
\label{strategyproof}
\subsubsection{Theorem \ref{mainirr}}
\label{mainone}We compute the dimension of the locally closed substacks $\mathscr{N}_{\bm{\alpha},\bm{\xi}}$ of $\mathscr{N}_{\bm{\alpha}}$ in two steps. First, we assume that $\bm{\alpha}=(\alpha)$, so that we only deal with semistable sheaves of fixed rank and degree. By the classification due Atiyah of vector bundles over an elliptic curve, such coherent sheaves form an algebraic stack isomorphic to the stack of torsion sheaves of degree $\gcd({\alpha})$. The case of torsion sheaves corresponds to the Springer situation for which the description of the irreducible components is well-known. Then, we use the morphism
\[
 p_{\bm{\alpha}} : \mathfrak{Coh}_{\bm{\alpha}}(X)\rightarrow \prod_{i=1}^s\mathfrak{Coh}_{(\alpha_i)}(X)
\]
which sends a coherent sheaf of HN-type $\bm{\alpha}$ to the collection of semistable factors of the Harder-Narasimhan filtration. This morphism has very favourable properties (it is a vector bundle stack (\cite{MR3293805,MR2139694})-- or more precisely an iteration of vector bundle stacks), and the dimension of the fiber is easy to compute. This allows us to perform efficiently the computation of the dimension of $\mathscr{N}_{\bm{\alpha},\bm{\xi}}$. Combined with the irreducibility of $\mathfrak{Coh}_{\bm{\alpha},\bm{\xi}}(X)$, this will allow us to prove Theorem \ref{mainirr}. Our strategy is inspired by the work of Ringel who determined the irreducible components of Lusztig nilpotent variety for affine quivers (\cite{MR1648647,MR1676227}).

\subsubsection{Theorem \ref{mainbij}}
\label{maintwo}
Let $\alpha\in\Z^+$ and $\bm{\alpha}\in HN(\alpha)$. We let $\mathfrak{Coh}_{\geq \bm{\alpha}}(X)$ be the open substack of $\mathfrak{Coh}_{\alpha}(X)$ which is the union of the $\mathfrak{Coh}_{\bm{\beta}}(X)$ such that $\mathfrak{Coh}_{\bm{\alpha}}(X)\subset \overline{\mathfrak{Coh}_{\bm{\beta}}}(X)$. We let $j_{\bm{\alpha}} : \mathfrak{Coh}_{\geq \bm{\alpha}}(X)\rightarrow\mathfrak{Coh}_{\alpha}(X)$ be the (open) inclusion. Since we work with completions, using the poset structure on the set of Harder-Narasimhan strata (endowed with the containment order) and letting
\[
 \mathcal{P}^{\bm{\alpha}}=\{(j_{\bm{\alpha}})^*\mathscr{F} : \mathscr{F}\in\Ob(\mathcal{P}^{\alpha}), \supp\mathscr{F}=\overline{\supp\mathscr{F}\cap\mathfrak{Coh}_{\bm{\alpha}}(X)}\},
 \]
it suffices to show that the characteristic cycle map
\[
 CC : K_0(\mathcal{P}^{\bm{\alpha}})\rightarrow \Z[\Irr(\mathscr{N}_{\bm{\alpha}})]
 \]
is a unitriangular isomorphism. We proceed in two steps by first assuming $\bm{\alpha}=(\alpha)$ so that we work on the semistable locus. There we use again the structure of coherent sheaves on an elliptic curve which allows us to consider only torsion sheaves. Last, using the explicit description of the simple ojects of $\mathcal{P}^{\bm{\alpha}}$, and the morphism $p_{\bm{\alpha}}$, we are able to give a rather concrete description of the characteristic cycle of any perverse sheaf of the category $\mathcal{P}^{\bm{\alpha}}$. Proving the isomorphism is now easy.

    \subsubsection{Theorem \ref{microlocchar}} The key fact, which is true not only for elliptic curves but also for curves of genus $g\geq 2$, is that any local system on some connected component of the semisimple locus extends to the whole connected component of $\mathfrak{Coh}(X)$. This can be proved by codimension considerations and by using the determinant morphism to the Picard stack, which provides a retraction to the inclusion of the Picard stack in the stack of rank one coherent sheaves. This allows us to prove that any induction of perverse sheaves as in Theorem \ref{microlocchar} has nilpotent singular support. The other implication of Theorem \ref{microlocchar} does not present difficulties and follows from the consideration of the restriction of the induction diagram to the sheaves of HN-type $\bm{\alpha}\in HN(\alpha)$ \eqref{iteratedind}.

\subsection{Contents of the paper}
\label{contents}
In Section \ref{Cohell}, we recall some standard facts concerning the stack of coherent sheaves on curves with an emphasis on the case of elliptic curves. We follow the presentation of Schiffmann (\cite{MR2942792}). We define stratifications of the semistable loci and glue them using the Harder-Narasimhan stratification to obtain a refined stratification of any Harder-Narasimhan stratum. In Section \ref{globnil}, we introduce the stack of Higgs sheaves and the global nilpotent cone. We give a description of the irreducible components of the global nilpotent cone and of its semistable locus for an elliptic curve. A partial order on the set of its irreducible components is defined combining the natural order on Harder-Narasimhan strata given by the dominance order and the antidominance order on partitions. We investigate in small ranks the problem of the relation of the two parametrizations of the irreducible components of the global nilpotent cone of an elliptic curve: the parametrization of Bozec and the one described here. This underlies nontrivial combinatorics and piecewise linear structures. In Section \ref{eisensteinsheaves}, we recall the induction and restriction functors on the derived category of constructible sheaves and the definition of spherical Eisenstein sheaves. In this generality, these functors appear in the work of Schiffmann (\cite{MR2942792}). They are defined in analogy with the quiver induction and restriction functors considered by Lusztig (\cite{MR1088333}) and with the induction functor studied by Laumon (\cite{MR899400,MR1044822}).  We recall the explicit description of simple Eisenstein perverse sheaves in terms of a local system on a smooth part of their support due to Schiffmann. This description allows us to describe the multiplicities of the irreducible components of the global nilpotent cone in the singular support of the restriction of a simple Eisenstein sheaf to its supporting Harder-Narasimhan stratum. Perverse sheaves on the nilpotent cone of $\mathfrak{gl}_n$ and the Kostka numbers appear in this description. In Section \ref{Proofs}, we detail the proof of the first two main theorems, namely the description of the irreducible components of the global nilpotent cone (Theorem \ref{mainirr}) and the lower unitriangularity of the characteristic cycle map from the category of spherical Eisenstein sheaves (Theorem \ref{mainbij}). In Section \ref{nilsingps}, we introduce the material needed to prove the third main result, Theorem \ref{microlocchar}, which gives an explicit description of the simple perverse sheaves on the stack of coherent sheaves on an elliptic curve having singular support in the global nilpotent cone. Along the way, we determine codimension one Harder-Narasimhan strata and show that local systems on the semi-stable locus always extend to the whole stack of coherent sheaves, for any curve.

\subsection{Notations and conventions}
\label{notations}
We let $\mathscr{P}$ be the set of partitions of positive integers. If $j\geq 1$, the unique partition of $j$ of length one is $(j)$. The set $\mathscr{P}$ is naturally ordered by the dominance order. If $d\geq 1$, $\mathscr{P}_d$ is the subset of partitions of $d$. If $\lambda\in\mathscr{P}_d$,  $\lvert\lambda\rvert=d$ and $l(\lambda)$ is the length of $\lambda$ (the number of nonzero parts). If $\bm{d}=(d_1,\hdots,d_s)$ is a $s$-uplet, $\mathscr{P}_{\bm{d}}=\{(\lambda_1,\hdots,\lambda_s)\in\mathscr{P}^s\mid \lambda_i\in\mathscr{P}_{d_i}\}$. We let $\N^{\mathscr{P}}$ be the set of functions $\mathscr{P}\rightarrow\N$. For $d\in\N$, $(\N^{\mathscr{P}})_d$ is the set of functions $\xi\in\N^{\mathscr{P}}$ such that $\sum_{\lambda\in\mathscr{P}}\xi(\lambda)\lvert\lambda\rvert=d$. If $\xi\in(\N^{\mathscr{P}})_d$ is such that for any $\lambda\in\mathscr{P}$, $\xi(\lambda)\neq 0\implies l(\lambda)=1$, the datum of $\xi$ is equivalent to the datum of the partition $\lambda=(j^{\xi((j))}:j\geq 1)\in\mathscr{P}_{d}$. In this case, we write $\xi=\xi_{\lambda}$. If $\bm{d}=(d_1,\hdots,d_s)$ is a $s$-uplet, $(\N^{\mathscr{P}})_{\bm{d}}$ is the set of $s$-uplets of functions $(\xi_1,\hdots,\xi_s)$ such that for $1\leq i\leq s$, $\xi_i\in(\N^{\mathscr{P}})_{d_i}$. If for any $1\leq i\leq s$ and any $\lambda\in\mathscr{P}$, $\xi_i(\lambda)\neq 0\implies l(\lambda)=1$, the datum of $\bm{\xi}$ is equivalent to the datum of the $s$-uplet of partitions $\bm{\lambda}=(\lambda_1,\hdots,\lambda_s)\in\mathscr{P}_{\bm{d}}$, where $\lambda_i=(j^{\xi_i((j))}:j\geq 1)$. In this case, we write $\bm{\xi}=\bm{\xi}_{\bm{\lambda}}$. We denote $\mathfrak{Coh}(X)$ the stack of coherent sheaves on a given fixed curve $X$. We let $\Z^+=\{(r,d)\in\Z^2\mid r>0 \text{ or }r=0 \text{ and }d>0\}$. For $\alpha\in\Z^+$, $\mathfrak{Coh}_{\alpha}(X)$ is the substack of coherent sheaves of class $\alpha$ and $\mu(\alpha)=\frac{d}{r}$ if $\alpha=(r,d)$. The set of Harder-Narasimhan types of class $\alpha$ is $HN(\alpha)$. Its elements are uplets $\bm{\alpha}=(\alpha_1,\hdots,\alpha_s)$ for some $s\geq 1$, $\alpha_i\in\Z^+$, $\mu(\alpha_1)>\hdots>\mu(\alpha_s)$, $\sum_{i=1}^s\alpha_i=\alpha$. For $\alpha=(r,d)\in\Z^2$, $\delta=\gcd(\alpha)=\gcd(r,d)$. For $\bm{\alpha}=(\alpha_1,\hdots,\alpha_s)\in(\Z^2)^s$, $\bm{\delta}=\gcd(\bm{\alpha})=(\gcd(\alpha_1),\hdots,\gcd(\alpha_s))$. For $\alpha\in\Z^+$ and $\bm{\alpha}\in HN(\alpha)$, $\mathfrak{Coh}_{\bm{\alpha}}(X)$ denotes the corresponding Harder-Narasimhan stratum. We will consider the derived category of constructible sheaves on the stack of coherent sheaves on a smooth projective curve. Each connected component $\mathfrak{Coh}_{\alpha}(X)$ of this stack has a presentation as an increasing union of open substacks which are quotient stacks, $\mathfrak{Coh}_{\alpha}(X)=\varinjlim_{n}\limits X_n/G_n$ where $X_n$ is an open subvariety of some Quot scheme and $G_n$ some (reductive) algebraic group. The derived category of the Artin stack $\mathfrak{Coh}_{\alpha}(X)$ can be dealt with using the formalism of equivariant derived categories (\cite{MR1299527}): $D_c(\mathfrak{Coh}_{\alpha}(X))=\varprojlim_n\limits D_{c,G_n}(X_n)$. If $X$ is any algebraic variety and $n\geq 1$ an integer, we let $\Delta\subset X^n$ and $\Delta\subset S^nX$ be the big diagonals (the closed subvarieties of $n$-uplets $(x_1,\hdots,x_n)$ such that two or more coordinates are equal).

\section{The moduli stack of coherent sheaves on elliptic curves}
\label{Cohell}

\subsection{Coherent sheaves on a curve}
\label{cohcurve}
We recall here the fundamental properties of the stack of coherent sheaves on a curve that we will need. We consider a smooth projective curve $X$ over the field of complex numbers $\C$. We let $\Coh(X)$ be the category of coherent sheaves on $X$. This is an abelian category of homological dimension one. To a coherent sheaf $\mathcal{F}$, we can assign its rank $r$ and its degree $d$. The pair $\alpha=(r,d)$ belongs to
\[
 \Z^+=\{(r,d)\in\Z^2\mid r>0 \text{ or }r=0, d>0\}.
\]
and is called the \emph{type} of $\mathcal{F}$. This assignment yields a group homomorphism
\[
 K_0(\Coh(X))\rightarrow \Z^2.
\]
For a coherent sheaf $\mathcal{F}$ on $X$, we let $[\mathcal{F}]\in\Z^+$ be the corresponding pair. The Euler form of the category $\Coh(X)$ factors through this morphism: for any coherent sheaves $\mathcal{F}, \mathcal{G}$ on $X$ such that $[\mathcal{F}]=(r_1,d_1)$ and $[\mathcal{G}]=(r_2,d_2)$,
\begin{equation}
\label{eulerform}
 \langle\mathcal{F},\mathcal{G}\rangle=\dim\Hom(\mathcal{F},\mathcal{G})-\dim\Ext^1(\mathcal{F},\mathcal{G})=(1-g)r_1r_2+(r_1d_2-r_2d_1)
\end{equation}
if $g$ is the genus of $X$, thanks to Serre duality and the Riemann-Roch formula. When $X$ is an elliptic curve, the first term disappears. There is a smooth stack $\mathfrak{Coh}(X)$ parametrizing the objects of $\Coh(X)$. It has an infinite number of connected components indexed by $\Z^+$:
\[
 \mathfrak{Coh}(X)=\bigsqcup_{\alpha\in\Z^+}\mathfrak{Coh}_{\alpha}(X)
\]
where the connected component $\mathfrak{Coh}_{\alpha}(X)$ is an Artin stack locally of finite type which parametrizes coherent sheaves on $X$ of class $\alpha$. The dimension of $\mathfrak{Coh}_{\alpha}(X)$ is $r^2(g-1)$ if $\alpha=(r,d)$. In particular, it is of dimension $0$ when $X$ is an elliptic curve.
\subsection{The category and moduli stack of torsion sheaves}
\label{torstack}

\subsubsection{Torsion sheaves on a curve}
The moduli stack of torsion sheaves is the union $\mathfrak{Tor}(X)$ of the connected components $\mathfrak{Coh}_{(0,d)}(X)$, $d>0$ of $\mathfrak{Coh}(X)$. Fix $d> 0$. There is a support map:
\[
\chi : \mathfrak{Coh}_{(0,d)}(X)\rightarrow S^dX
\]
to the $d$-th symmetric power of $X$ sending a coherent sheaf to its support (with multiplicities).

The category of torsion sheaves on $X$ is denoted $\Tor(X)$. For any closed point $x\in X$, we let $\Tor_x(X)$ be the subcategory of torsion sheaves supported at $x$. It is equivalent to the category of nilpotent representations of the Jordan quiver (the quiver with one vertex and one loop). In other words, the datum of a torsion sheaf of degree $d$ supported at a given point is the same thing as the datum of a nilpotent $d\times d$ matrix. In particular, isomorphism classes of torsion sheaves supported at $x$ are indexed by partitions according to the sizes of the Jordan blocks. We let $T_{x,\lambda}$ be the torsion sheaf associated to the partition $\lambda\in\mathscr{P}$.

\subsubsection{Stratification of the stack of torsion sheaves on a curve}
\label{strattorsion}
Let $d\geq 0$. The stack $\mathfrak{Coh}_{(0,d)}(X)$ admits a stratification indexed by the set $(\N^{\mathscr{P}})_{d}$ of functions $\xi : \mathscr{P}\rightarrow \N$ satisfying $\sum_{\lambda\in\mathscr{P}}\xi(\lambda)\lvert\lambda\rvert=d$. Namely, we write
\[
 \mathfrak{Coh}_{(0,d)}(X)=\bigsqcup_{\xi\in(\N^{\mathscr{P}})_d}\mathfrak{Coh}_{(0,d),\xi}(X)
\]
where $\mathfrak{Coh}_{(0,d),\xi}(X)$ is defined as follows. We pick partitions $\lambda_1,\hdots,\lambda_s\in\mathscr{P}$ so that in this collection, any partition $\lambda\in\mathscr{P}$ appears precisely $\xi(\lambda)$ times. Then, geometric points of $\mathfrak{Coh}_{(0,d),\xi}(X)$ parametrize by definition isomorphism classes of torsion sheaves $T$ on $X$ isomorphic to a direct sum
\[
 \bigoplus_{i=1}^sT_{x_i,\lambda_i}
\]
for pairwise distinct points $x_i\in X$. This stratification is analogous to the Jordan stratification of $\mathfrak{gl}_d$ (the case when $X=\A^1$) and also to the stratification of the regular locus of the representation stack of affine quivers (\cite{MR1676227}) which has been of great help to the author in the understanding of perverse sheaves with nilpotent singular support (see \cite{2020henn}).
\subsection{Category and stack of semistable sheaves of fixed rank and degree}
\label{sstable}
Let $\alpha=(r,d)\in\Z^+$. For a nonzero coherent sheaf $\mathcal{F}$ on $X$ of rank $r$ and degree $d$, we define its slope $\mu(\mathcal{F})=\frac{d}{r}$ with the convention that the slope of a torsion sheaf is infinite. We say that $\mathcal{F}$ is semistable if for any nonzero proper subsheaf $\mathcal{G}\subset \mathcal{F}$, $\mu(\mathcal{G})\leq \mu(\mathcal{F})$. Stable sheaves are those $\mathcal{F}$ for which the inequality is always strict. Given a fixed slope $\mu\in\Q$, the category of semistable coherent sheaves on $X$ of slope $\mu$ is abelian, noetherian and artinian. There is a moduli stack of semistable coherent sheaves of slope $\mu$, $\mathfrak{Coh}^{\mu}(X)$. It is an open substack of $\mathfrak{Coh}(X)$. If $\alpha=(r,d)\in\Z^+$ is such that $\mu=\frac{d}{r}$ and $\gcd(r,d)=1$, we have a decomposition in connected components of $\mathfrak{Coh}^{\mu}(X)$:
\[
 \mathfrak{Coh}^{\mu}(X)=\bigsqcup_{l\geq 1}\mathfrak{Coh}_{(l\alpha)}(X),
\]
where for any $\beta\in\Z^+$, $\mathfrak{Coh}_{(\beta)}(X)$ is the open subtack of $\mathfrak{Coh}_{\beta}(X)$ classifying semistable coherent sheaves of class $\beta$. If $X$ is an elliptic curve, there is an equivalence of categories between $\Tor(X)=\Coh^{\infty}(X)$ and $\Coh^{\mu}(X)$, which can be constructed using mutations or equivalently, Fourier-Mukai transforms (see \cite[\S 1.1 d)]{MR2942792}, \cite{MR1074778}). We let $\mathfrak{Tor}(X)$ be the stack of torsion sheaves on $X$. It coincides with $\mathfrak{Coh}^{\infty}(X)$. The equivalence sends semistable sheaves of rank $r$ and degree $d$ with $\mu=\frac{d}{r}$ to torsion sheaves of degree $\delta=\gcd(r,d)$. It also induces isomorphisms at the level of stacks
\[
 \epsilon_{\mu} : \mathfrak{Tor}(X)\rightarrow \mathfrak{Coh}^{\mu}(X)
\]
and
\[
 \epsilon_{\alpha} : \mathfrak{Coh}_{(0,\delta)}(X)\rightarrow \mathfrak{Coh}_{(\alpha)}(X).
 \]
Thanks to the isomorphism $\epsilon_{\alpha}$, we transport the stratification of $\mathfrak{Coh}_{(0,\delta)}(X)$ to $\mathfrak{Coh}_{(\alpha)}(X)$:
\[
 \mathfrak{Coh}_{(\alpha)}(X)=\bigsqcup_{\xi\in(\N^{\mathscr{P}})_{\delta}}\mathfrak{Coh}_{(\alpha),\xi}(X)
\]
where $\mathfrak{Coh}_{(\alpha),\xi}(X)=\epsilon_{\alpha}(\mathfrak{Coh}_{(0,\delta),\xi}(X))$.

We have the following properties for semistable sheaves of different slopes which will be useful for understanding the induction of perverse sheaves. If $\nu>\mu\in\Q$ are two slopes, $\mathcal{F}\in\Coh^{\mu}(X)$ and $\mathcal{G}\in\Coh^{\nu}(X)$, then
\[
 \Hom_{\OO_X}(\mathcal{G},\mathcal{F})=0.
\]
Since the canonical bundle of an elliptic curve is trivial, Serre duality implies that
\begin{equation}
\label{extvan}
 \Ext^1_{\OO_X}(\mathcal{F},\mathcal{G})=0.
\end{equation}

\subsection{A refinement of the Harder-Narasimhan stratification}
\label{refstrat}

\subsubsection{The Harder-Narasimhan stratification}
\label{hnstrat}
Let $\mathcal{F}$ be a coherent sheaf on $X$. It admits a unique filtration $0=\mathcal{F}_0\subset \mathcal{F}_1\subset \hdots\subset \mathcal{F}_s=\mathcal{F}$ such that for any $1\leq i\leq s$, the quotient $\mathcal{F}_i/\mathcal{F}_{i-1}$ is semistable and if $\mu_i$ denotes its slope, $\mu_1>\mu_2>\hdots>\mu_s$. It is called the Harder-Narasimhan stratification of $\mathcal{F}$. The Harder-Narasimhan type of $\mathcal{F}$ is then the collection $([\mathcal{F}_i/\mathcal{F}_{i-1}])_{1\leq i\leq s}$. We let $HN(\alpha)$ be the set of all possible Harder-Narasimhan types of coherent sheaves of type $\alpha$. For $\alpha\in\Z^+$, the Harder-Narasimhan stratification of $\mathfrak{Coh}_{\alpha}(X)$ is then
\[
 \mathfrak{Coh}_{\alpha}(X)=\bigsqcup_{\bm{\alpha}\in HN(\alpha)}\mathfrak{Coh}_{\bm{\alpha}}(X),
\]
where $\mathfrak{Coh}_{\bm{\alpha}}(X)$ is the locally closed substack parametrizing coherent sheaves on $X$ of HN-type $\bm{\alpha}$. In particular, if $\bm{\alpha}=(\alpha)$ for some $\alpha\in\Z^+$, then the notation $\mathfrak{Coh}_{\bm{\alpha}}(X)=\mathfrak{Coh}_{(\alpha)}(X)$ is the one introduced in Section \ref{sstable}.

\subsubsection{Refinement of the Harder-Narasimhan stratification for elliptic curves}
\label{refhnstrat}
Let $\alpha=(r,d)\in\Z^+$ and $\bm{\alpha}=(\alpha_1,\hdots,\alpha_s)\in HN(\alpha)$. There is a morphism of stacks
\[
 p_{\bm{\alpha}} : \mathfrak{Coh}_{\bm{\alpha}}(X)\rightarrow\prod_{i=1}^s\mathfrak{Coh}_{(\alpha_i)}(X)
\]
which sends a coherent sheaf of HN-type $\bm{\alpha}$ to the collection of the subquotients of the Harder-Narasimhan filtration. If $X$ is an elliptic curve, the Harder-Narasimhan filtration splits (noncanonically): if $\mathcal{F}$ is a coherent sheaf on $X$ and $(\mathcal{F}_i)_{1\leq i\leq s}$ is its Harder-Narasimhan filtration,
\[
 \mathcal{F}\simeq \bigoplus_{i=1}^s\mathcal{F}_i/\mathcal{F}_{i-1}.
\]
This is a consequence of the Ext-vanishing property \eqref{extvan}. Therefore, the fiber of $p_{\bm{\alpha}}$ over a $\C$-point $(\mathcal{F}_1,\hdots,\mathcal{F}_s)$ can be identified with the stack quotient
\[
 \pt/\bigoplus_{j<i}\Hom_{\OO_X}(\mathcal{F}_i/\mathcal{F}_{i-1},\mathcal{F}_j/\mathcal{F}_{j-1}).
\]
In particular, its dimension is
\begin{equation}
\label{dimfiberp}
  d_{\bm{\alpha}}=-\sum_{j<i}(r_id_j-r_jd_i),
\end{equation}
where $\alpha_i=(r_i,d_i)=[\mathcal{F}_i/\mathcal{F}_{i-1}]$.

We define a stratification of $\mathfrak{Coh}_{\bm{\alpha}}(X)$:
\[
 \mathfrak{Coh}_{\bm{\alpha}}(X)=\bigsqcup_{\bm{\xi}\in (\N^{\mathscr{P}})_{\bm{\delta}}}\mathfrak{Coh}_{\bm{\alpha},\bm{\xi}}(X)
\]
indexed by $(\N^{\mathscr{P}})_{\bm{\delta}}=\{\bm{\xi}=(\xi_1,\hdots,\xi_s)\in(\N^{\mathscr{P}})^s\mid \sum_{\lambda\in\mathscr{P}}\xi_i(\lambda)\lvert\lambda\rvert=\gcd(\alpha_i)\text{ for any $1\leq i\leq s$}\}$, where
\[
 (\mathfrak{Coh}_{\bm{\alpha},\bm{\xi}}(X))=p_{\bm{\alpha}}^{-1}\left(\prod_{i=1}^s\mathfrak{Coh}_{(\alpha_i),\xi_i}(X)\right).
\]
In concrete terms, the $\C$-points of $\mathfrak{Coh}_{\bm{\alpha},\bm{\xi}}(X)$ are the isomorphism classes of coherent sheaves on $X$ having HN-type $\bm{\alpha}$ and such that the $i$-th subquotient of the HN-filtration belongs to $\mathfrak{Coh}_{(\alpha_i),\xi_i}(X)$.

\subsection{$\SL_2(\Z)$-action on the stack of coherent sheaves}
\label{slaction}
It is a well-known fact that Harder-Narasimhan strata of $\mathfrak{Coh}(X)$ can be indexed by convex path in $\Z^+$ starting from the origin. Namely, for $\alpha\in\Z^+$ and $\bm{\alpha}=(\alpha_1,\hdots,\alpha_s)\in HN(\alpha)$, we take the piecewise affine path $\bigsqcup_{i=0}^{s-1}[\sum_{j=0}^{i-1}\alpha_{s-j},\sum_{j=0}^i\alpha_{s-j}]$ with the convention that $\alpha_0=0$. We denote $\bm{p}_{\bm{\alpha}}$ this path. The convexity comes from the condition on the successive slopes of a Harder-Narasimhan type. The inclusion of Harder-Narasimhan strata can be described in terms of convex paths: for any $\bm{\alpha},\bm{\beta}\in HN(\alpha)$, $\mathfrak{Coh}_{\bm{\beta}}(X)\subset \overline{\mathfrak{Coh}_{\bm{\alpha}}(X)}$ if and only if $\bm{p}_{\bm{\beta}}$ lies below $\bm{p}_{\bm{\alpha}}$. This defines the order relation $\mathfrak{Coh}_{\bm{\beta}}(X)\subset \overline{\mathfrak{Coh}_{\bm{\alpha}}(X)}\iff \bm{\beta}\geq\bm{\alpha}$ on $HN(\alpha)$ (the dense stratum is the smallest element of $HN(\alpha)$). The group $\SL_2(\Z)$ acts naturally on $\Z^2$ by $\Z$-linear automorphisms. Therefore, it acts on tuples $(\alpha_1,\hdots,\alpha_s)$ by $\gamma\cdot (\alpha_1,\hdots,\alpha_s)=(\gamma\cdot\alpha_1,\hdots,\gamma\cdot\alpha_s)$. If $A\subset HN(\alpha)$ is some set of Harder-Narasimhan types such that for any $\bm{\beta}\in HN(\alpha)$ and $\bm{\alpha}\in A$, $\bm{\beta}\leq \bm{\alpha}$ implies $\bm{\beta}\in A$ and $\gamma\in \SL_2(\Z)$ is such that $\gamma\cdot A=\{\gamma\cdot\bm{\alpha} : \bm{\alpha}\in A\}\subset HN(\gamma\cdot\alpha)$ (this is the condition that $\gamma$ sends $\bm{p}_{\bm{\alpha}}$ to a path contained in $\Z^+$), then $\gamma$ induces an isomorphism between the open substacks $\mathfrak{Coh}_A=\bigsqcup_{\bm{\alpha}\in A}\mathfrak{Coh}_{\bm{\alpha}}(X)$ of $\mathfrak{Coh}_{\alpha}(X)$ and $\mathfrak{Coh}_{\gamma\cdot A}=\bigsqcup_{\bm{\alpha}\in A}\mathfrak{Coh}_{\gamma\cdot\bm{\alpha}}(X)$ of $\mathfrak{Coh}_{\gamma\cdot\alpha}(X)$,
\begin{equation}
\label{isostra}
 i_{\gamma}:\mathfrak{Coh}_A\rightarrow \mathfrak{Coh}_{\gamma\cdot A}
\end{equation}
already used in the proof of \cite[Proposition 6.7]{MR2942792}.

\section{The global nilpotent cone}
\label{globnil}
In this section, we describe the irreducible components of the global nilpotent cone in terms of the stratification of Section \ref{refstrat}.

\subsection{Higgs sheaves and the global nilpotent cone}
\label{higgssheaves}
We refer to \cite[\S 2.3]{2018arXivss} for more details and general properties of the stack of Higgs sheaves. A Higgs sheaf on $X$ is a pair $(\mathcal{F},\theta)$ of a coherent sheaf $\mathcal{F}$ and a morphism $\theta : \mathcal{F}\rightarrow \mathcal{F}\otimes_{\OO_X}K$ where $K$ is the canonical bundle of $X$. The moduli stack of Higgs sheaves on $X$ is denoted $\mathfrak{Higgs}(X)$. It has an infinite number of connected components indexed by $\Z^+$:
\[
 \mathfrak{Higgs}(X)=\bigsqcup_{\alpha\in\Z^+}\mathfrak{Higgs}_{\alpha}(X)
\]
where $\mathfrak{Higgs}_{\alpha}(X)$ parametrizes those Higgs sheaves on $X$ whose underlying coherent sheaf is of class $\alpha$. There is a projection
\[
 \pi : \mathfrak{Higgs}(X)\rightarrow \mathfrak{Coh}(X)
\]
forgetting the Higgs field which is compatible with the decompositions into connected components. For $\alpha\in\Z^+$, $\pi_{\alpha} : \mathfrak{Higgs}_{\alpha}(X)\rightarrow \mathfrak{Coh}_{\alpha}(X)$ denotes the restriction of $\pi$. It is the cotangent stack of $\mathfrak{Coh}_{\alpha}(X)$. Moreover, $\mathfrak{Higgs}_{\alpha}(X)$ is locally of finite type, of dimension $-2\langle\alpha,\alpha\rangle$. In particular, it is of dimension $0$ for an elliptic curve. The global nilpotent cone is the closed substack $\mathscr{N}$ parametrizing nilpotent Higgs sheaves, that is pairs $(\mathcal{F},\theta)$ such that the composition
\[
 \mathcal{F}\xrightarrow{\theta}\mathcal{F}\otimes K\xrightarrow{\theta\otimes\id_K}\mathcal{F}\otimes K^{\otimes 2}\xrightarrow{\theta\otimes \id_{K^{\otimes 2}}}\hdots\xrightarrow{\theta\otimes \id_{K^{\otimes{(n-1)}}}} \mathcal{F}\otimes K^{\otimes n}
\]
vanishes for $n$ sufficiently large. We write it $\mathscr{N}=\bigsqcup_{\alpha\in\Z^+}\mathscr{N}_{\alpha}$. For any $\alpha\in\Z^+$, $\mathscr{N}_{\alpha}$ is a Lagrangian substack of $\mathfrak{Higgs}_{\alpha}(X)$ (\cite{MR962524,MR1853354}). The restriction of $\pi_{\alpha}$ to $\mathscr{N}_{\alpha}$ is denoted $\pi_{\alpha,\mathscr{N}}$.

\subsection{Irreducible components of the semistable nilpotent cone}
\label{irrcompnilconess}
We introduce the notion of stability for Higgs sheaves and show that for elliptic curve it coincides with the stability of the underlying coherent sheaf. Then, we give a parametrization of the irreducible components of the semistable locus of the elliptic global nilpotent cone (Corollary \ref{irrcompss}). For higher genus, see \cite{2017bozec}.

Let $(\mathcal{F},\theta)$ be a Higgs sheaf on $X$. We say that it is semistable if for any $0\subsetneq \mathcal{G}\subsetneq \mathcal{F}$ such that $\theta(\mathcal{G})\subset\mathcal{G}\otimes K$, $\mu(\mathcal{G})\leq \mu(\mathcal{F})$, and stable if these inequalities are always strict.

\begin{lemma}
\label{stablehiggs}
 Assume $X$ is an elliptic curve. A Higgs sheaf on $X$ is semistable (resp. stable) if and only if the underlying coherent sheaf is semistable (resp. stable).
\end{lemma}

\begin{proof}
 Let $(\mathcal{F},\theta)$ be a Higgs sheaf on $X$. The canonical bundle of an elliptic curve is trivial, so $\theta\in\End{\mathcal{F}}$. Moreover, any endomorphism respects the Harder-Narasimhan of $\mathcal{F}$ so the condition of being a semistable or stable Higgs sheaf is the same as the condition on the underlying coherent sheaf of being semistable or stable.
\end{proof}
The same argument shows that the Harder-Narasimhan stratification of $\mathfrak{Higgs}_{\alpha}(X)$ is induced by the Harder-Narasimhan stratification of $\mathfrak{Coh}_{\alpha}(X)$ when $X$ is an elliptic curve: the former is the pull-back by $\pi_{\alpha}$ of the latter. For a HN-type $\bm{\alpha}$, we let $\mathfrak{Higgs}_{\bm{\alpha}}(X)$ be the locally closed substack of Higgs bundles of HN-type $\bm{\alpha}$. By Lemma \ref{stablehiggs}, $\mathfrak{Higgs}_{\bm{\alpha}}(X)=\pi_{\alpha}^{-1}(\mathfrak{Coh}_{\bm{\alpha}}(X))$. We let $\mathscr{N}_{\bm{\alpha}}$ be the union of the irreducible components of dimension $\dim\mathscr{N}$ of $\mathfrak{Higgs}_{\bm{\alpha}}(X)\cap\mathscr{N}_{\alpha}$. In particular, $\mathscr{N}_{(\alpha)}$ denotes the stack of nilpotent semistable Higgs bundles of class $\alpha$ on an elliptic curve $X$. By Lemma \ref{stablehiggs}, we have
\[
 \mathscr{N}_{(\alpha)}=\pi_{\alpha,\mathscr{N}}^{-1}(\mathfrak{Coh}_{(\alpha)}(X)).
\]
Therefore, by Theorem \ref{mainirr}, the irreducible components of $\mathscr{N}_{(\alpha)}$ can be described as in Corollary \ref{irrcompss} below.

The semistable nilpotent cone is the union
\[
 \mathscr{N}^{ss}=\bigsqcup_{\alpha\in\Z^+}\mathscr{N}_{(\alpha)}
\]
Since $\mathfrak{Coh_{(\alpha)}}(X)$ is open in $\mathfrak{Coh}_{\alpha}(X)$, $\mathscr{N}_{(\alpha)}$ is a Lagrangian substack of $\mathscr{N}_{(\alpha)}=\pi_{\alpha,\mathscr{N}}^{-1}(\mathfrak{Coh}_{(\alpha)}(X))$. For $\xi\in(\N^\mathscr{P})_{\delta}$ ($\delta=\gcd(\alpha)$), we let $\mathscr{N}_{(\alpha),\xi}=\pi_{\alpha,\mathscr{N}}^{-1}(\mathfrak{Coh}_{(\alpha),\xi}(X))$. We let $\overline{\mathscr{N}_{(\alpha),\xi}}$ be the closure of $\mathscr{N}_{(\alpha),\xi}$ in $\mathscr{N}_{(\alpha)}$ and $\Overline[1.5]{\mathscr{N}_{(\alpha),\xi}}$ be the closure of $\mathscr{N}_{(\alpha),\xi}$ in $\mathscr{N}_{\alpha}$. If $\xi=\xi_{\lambda}$ for some partition $\lambda\in\mathscr{P}_{\delta}$, we write $\mathscr{N}_{\alpha,\lambda}=\mathscr{N}_{\alpha,\xi}$. We can now formulate a corollary of Theorem \ref{mainirr}.

\begin{cor}
\label{irrcompss}
 The irreducible components of $\mathscr{N}_{(\alpha)}$ are the closed substacks $\overline{\mathscr{N}_{(\alpha),\lambda}}$ for $\lambda\in(\N^{\mathscr{P}})_{\delta}$, $\delta=\gcd{\alpha}$.
\end{cor}
\begin{proof}
 This corollary is a consequence of Lemma \ref{stablehiggs} and Theorem \ref{mainirr} for $\bm{\alpha}=(\alpha)$ which we prove in Section \ref{proofthm1}
\end{proof}

\subsection{Irreducible components of the global nilpotent cone}
\label{ircompnc}
Let $\alpha\in\Z^+$ and $\bm{\alpha}\in HN(\alpha)$. We let $\mathscr{N}_{\bm{\alpha}}$ be the union of the irreducible components of dimension $\dim\mathscr{N}$ of $\pi_{\alpha,\mathscr{N}}^{-1}(\mathfrak{Coh}_{\bm{\alpha}}(X))$. If $\bm{\alpha}=(\alpha_1,\hdots,\alpha_s)$ and $\bm{\xi}=(\xi_1,\hdots,\xi_s)\in(\N^{\mathscr{P}})_{\bm{\delta}}$, we let $\mathscr{N}_{\bm{\alpha},\bm{\xi}}=\pi_{\alpha,\mathscr{N}}^{-1}(\mathfrak{Coh}_{\bm{\alpha},\bm{\xi}}(X))$. In particular, for $\bm{\alpha}=(\alpha)$ and $\bm{\xi}=\xi\in\mathscr{P}_{\delta}$, we obtain the locally closed substacks $\mathscr{N}_{(\alpha),\xi}$ of $\mathscr{N}_{(\alpha)}$ defined in Section \ref{irrcompnilconess}. Also, $\overline{\mathscr{N}_{\bm{\alpha},\bm{\xi}}}$ is the closure of $\mathscr{N}_{\bm{\alpha},\bm{\xi}}$ in $\mathscr{N}_{\bm{\alpha}}$ and $\Overline[1.5]{\mathscr{N}_{\bm{\alpha},\bm{\xi}}}$ is the closure of $\mathscr{N}_{\bm{\alpha},\bm{\xi}}$ in $\mathscr{N}_{\alpha}$. If $\bm{\xi}=\xi_{\bm{\lambda}}$ for some uplet of partitions $\bm{\lambda}\in\mathscr{P}_{\bm{\delta}}$, we let $\mathscr{N}_{\bm{\alpha},\bm{\lambda}}=\mathscr{N}_{\bm{\alpha},\bm{\xi}}$.

Let $\Lambda\subset \mathscr{N}_{\alpha}$ be an irreducible component. We call \emph{supporting stratum} of $\Lambda$ the unique HN-stratum $S$ of $\mathfrak{Coh}_{\alpha}(X)$ such that $\overline{S\cap \pi_{\alpha}(\Lambda)}=\pi_{\alpha}(\Lambda)$. We reformulate here Theorem \ref{mainirr} whose proof will be given in Section \ref{proofthm1}.
\begin{theorem}
\label{irrcomphn}
 Let $\alpha\in\Z^+$ and $\bm{\alpha}\in HN(\alpha)$. Irreducible components of $\mathscr{N}_{\alpha}$ whose supporting stratum is $\mathfrak{Coh}_{\bm{\alpha}}(X)$ are the irreducible components of $\overline{\mathscr{N}_{\bm{\alpha}}}\subset \mathscr{N}_{\alpha}$. They are the $\Overline[1.5]{\mathscr{N}_{\bm{\alpha},\bm{\lambda}}}$ for $\bm{\lambda}\in\mathscr{P}_{\bm{\delta}}$.
\end{theorem}

\subsection{Partial order on the set of irreducible components of the global nilpotent cone}
We define a partial order on the set of irreducible components of the global nilpotent cone as follows. If $\Overline[1.5]{\mathscr{N}_{\bm{\alpha},\bm{\lambda}}}$ and $\Overline[1.5]{\mathscr{N}_{\bm{\beta},\bm{\nu}}}$ are two irreducible components of $\mathscr{N}_{\alpha}$. We say that $\Overline[1.5]{\mathscr{N}_{\bm{\alpha},\bm{\lambda}}}\leq \Overline[1.5]{\mathscr{N}_{\bm{\beta},\bm{\nu}}}$ if $\mathfrak{Coh}_{\bm{\beta}}(X)$ is strictly contained in $\overline{\mathfrak{Coh}_{\bm{\alpha}}(X)}$ or $\bm{\alpha}=\bm{\beta}$ and $\lambda_i\geq \nu_i$ for any $1\leq i\leq s$ when we write $\bm{\lambda}=(\lambda_1,\hdots,\lambda_s)$ and $\bm{\nu}=(\nu_1,\hdots,\nu_s)$. We order similarly the set of irreducible components $\Irr(\mathscr{N}_{\bm{\alpha}})=\{\overline{\mathscr{N}_{\bm{\alpha},\bm{\lambda}}}:\bm{\lambda}\in(\N^{\mathscr{P}})_{\bm{\delta}}\}$ of $\mathscr{N}_{\bm{\alpha}}$ for $\bm{\alpha}\in HN(\alpha)$ (this order reduces to the antidominant order on the uplet of partitions $\bm{\lambda}$). We define a completion of the $\Z$-modules $\Z[\Irr(\mathscr{N}_{\alpha})]$ of functions $\Irr(\mathscr{N}_{\alpha})\rightarrow\Z$ with finite support. Namely, we let $\widehat{\Z[\Irr(\mathscr{N}_{\alpha})]}$ be the $\Z$-module of all functions $\Irr(\mathscr{N}_{\alpha})\rightarrow\Z$. We define in a similar way the completion $\widehat{\Z[\Irr(\mathscr{N}_{\bm{\alpha}})]}$ of $\Z[\Irr(\mathscr{N}_{\bm{\alpha}})]$.

\subsection{Comparison of the two parametrizations of the irreducible components}
In \cite{2017bozec}, Bozec describes a parametrization of the irreducible components of the global nilpotent cone (in arbitrary genus). We briefly recall his parametrization for elliptic curves. Let $\alpha\in\Z^+$. An $\alpha$-partition is a $s$-uplet $\underline{\alpha}=(\alpha_1,\hdots,\alpha_s)$ of elements of $\Z^+$ such that $\sum_{i=1}^si\alpha_i=\alpha$. If $(\mathcal{F},\theta)$ is a nilpotent Higgs sheaf with nilpotency order $s$, its Jordan-type is the $\alpha$-partition $\underline{\alpha}=(\alpha_1,\hdots,\alpha_s)$ such that
\[
 \alpha_k=[\ker(\im\theta^{k-1}/\im\theta^k\xrightarrow{\theta}\im\theta^k/\im\theta^{k+1})].
\]
We let $JT(\alpha)$ be the set of $\alpha$-partitions. For any $\alpha$-partition $\underline{\alpha}$, we let $\Lambda_{\underline{\alpha}}\subset \mathscr{N}_{\alpha}$ be the locally closed substack parametrizing nilpotent Higgs sheaves of Jordan-type $\underline{\alpha}$.
We have then the following theorem.
\begin{theorem}[{\cite[Corollary 2.5]{2017bozec}}]
\label{Bozecpart}
 We have a locally closed partition
 \[
  \mathscr{N}_{\alpha}=\bigsqcup_{\underline{\alpha}\in JT(\alpha)}\Lambda_{\underline{\alpha}}.
 \]
Moreover, the irreducible components of $\mathscr{N}_{\alpha}$ are the closures of the strata of this partition: $\Irr(\mathscr{N}_{\alpha})=\{\overline{\Lambda_{\underline{\alpha}}}:\underline{\alpha}\in JT(\alpha)\}$.
\end{theorem}
We have two parametrizations of the irreducible components of the global nilpotent cone given by Theorem \ref{irrcomphn} and Theorem \ref{Bozecpart}. An interesting and rather subtle question is to understand how they correspond to each other. The bijection between the parametrizations can be made quite explicit in small ranks but become much more intricate as the rank increases. In rank $0$, the two parametrizations amount to the parametrization by partitions of the length of a general torsion sheaf of the irreducible component and is similar to the parametrization of nilpotent orbits of $\mathfrak{gl}_d$ for $d\geq 0$. We can restrict ourselves to the case of vector bundles thanks to the following proposition (which holds for any curve).

\begin{proposition}
\label{propHiggsbundles}
 Let $X$ be a smooth projective curve, $\alpha\in\Z^+$ and $\underline{\alpha}=(\alpha_1,\hdots,\alpha_s)$ an $\alpha$-partition. Then, a general Higgs sheaf $(\mathcal{F},\theta)$ of the irreducible component  $\overline{\Lambda_{\underline{\alpha}}}$ is a Higgs bundle if and only if $\alpha_s$ is not a torsion class. Moreover, if $0\leq r\leq s$ is the biggest integer such that $\alpha_r$ is not a torsion class, the torsion part of a general Higgs sheaf of $\overline{\Lambda_{\underline{\alpha}}}$ has Jordan type $(\alpha_{r+1},\hdots,\alpha_s)$.
\end{proposition}
\begin{proof}
 If $\alpha_s$ is of rank $0$ and positive degree, since $\alpha_s=[\im\theta^{s-1}(-(s-1)\Omega_X)]$ where $\Omega_X$ is the dualizing sheaf of $X$ (\cite[\S 2]{2017bozec}), a Higgs sheaf with Jordan type $\underline{\alpha}$ has a torsion subsheaf. Conversely, let $\overline{\Lambda_{\underline{\alpha}}}$ be an irreducible component of the global nilpotent cone such that a general Higgs sheaf of $\overline{\Lambda_{\underline{\alpha}}}$ has a nontrivial torsion part. We have to show that $\alpha_s$ is torsion. By Theorem \ref{irrcomphn}, a general Higgs sheaf of $\overline{\Lambda_{\underline{\alpha}}}$ is of the form $(\mathcal{T}\oplus\mathcal{F},\theta)$ where $\mathcal{T}$ is a direct sum of indecomposable torsion sheaves supported at pairwise distinct points of $X$ and $\theta$ is a general nilpotent Higgs sheaf (this is true for curves of any genus). The Higgs field $\theta$ is of the form
 \[
  \theta=\begin{pmatrix}
          f&g\\
          0&h
         \end{pmatrix}
 \]
where $h\in\Hom(\mathcal{F},\mathcal{F}\otimes\mathcal{O}_X)$, $g\in\Hom(\mathcal{F},\mathcal{T})$ and $f\in\End(\mathcal{T})$. Let $0\leq r\leq s$ be the nilpotency order of $h$. Then, $\im h^{t}(-t\Omega_X)\subset \mathcal{F}$ is a vector bundle for $1\leq t\leq r-1$. By Lemma \ref{lemmasurjective}, if $g$ is general, $g_{|\im h^{t}(-t\Omega_X)}$ is surjective for $1\leq t\leq r-1$. Therefore, for $r\leq t\leq s$, $\im\theta^t(-t\Omega_X)=\im f^{t-r}\subset \mathcal{T}$. This proves the converse and the last assertion of the Proposition.
\end{proof}
\begin{lemma}
\label{lemmasurjective}
 Let $X$ be a smooth projective curve and $\mathcal{T}=\bigoplus_{i=1}^N\mathcal{T}_i$ be a direct sum of indecomposable torsion sheaves supported at pairwise distinct points. Then, for any vector bundle $\mathcal{F}$, a general morphism $\mathcal{F}\rightarrow \mathcal{T}$ is surjective. 
\end{lemma}
\begin{proof}
 It follows immediately from the fact that a vector bundle over a smooth projective curve is Zariski locally trivial, so we can assume $\mathcal{F}=\mathcal{O}_X$ and then the result is trivial.
\end{proof}

Thanks to Proposition \ref{propHiggsbundles}, we can make the link between the two parametrizations (given by Theorem \ref{irrcomphn} and Theorem \ref{Bozecpart}) in small ranks. It suffices to treat Higgs bundles.

\subsubsection{Rank 1} To illustrate Proposition \ref{propHiggsbundles}, we do not restrict ourselves to Higgs bundles yet. Let $\bm{\alpha}=((0,d_1),(1,d_2))\in HN(1,d_1+d_2)$ be a Harder-Narasimhan type of rank $1$ and $\bm{\lambda}=(\lambda_1,\lambda_2)$ where $\lambda_1$ is a partition of $d_1$ and $\lambda_2=(1)$ the unique partition of $\gcd{(1,d_2)}=1$. Then, we have
\[
 \overline{\Lambda_{\underline{\alpha}}}=\Overline{\mathscr{N}_{\bm{\alpha},\bm{\lambda}}}
\]
where $\underline{\alpha}=((1,d),(0,m_1),\hdots,(0,m_s))$ if $\lambda_1=(1^{m_1},\hdots,s^{m_s})$ and $d=d_1+d_2-\sum_{i=1}^sm_i$.

\subsubsection{Rank 2} Let $\alpha=(2,d)\in\Z^+$ and $\bm{\alpha}=(\alpha_1,\hdots,\alpha_s)\in HN(\alpha)$. Thanks to Proposition \ref{propHiggsbundles}, it suffices to treat the cases of Higgs bundles. The different cases are as follows.
\begin{enumerate}
 \item $s=1$, $d$ odd, $(\bm{\alpha},\bm{\lambda})=((\alpha),(1))$
 \item $s=1$, $d$ even, $(\bm{\alpha},\bm{\lambda})=((\alpha),(2))$,
 \item $s=1$, $d$ even, $(\bm{\alpha},\bm{\lambda})=((\alpha),(1,1))$,
 \item $s=2$, $(\bm{\alpha},\bm{\lambda})=(((1,d_1),(1,d_2)),((1),(1)))$ with $d_1>d_2$, $d_1+d_2=d$.
\end{enumerate}
The associated Jordan type are as follows.
\begin{enumerate}
 \item $(\alpha)$,
 \item $((0,0),(1,\frac{d}{2}))$,
 \item $(\alpha)$,
 \item $((0,d_1-d_2),(1,d_2))$.
\end{enumerate}

\subsubsection{Rank 3}From rank $3$, the situation becomes more subtle. In particular, we need to be able to compute the general type of the image of a morphism from a rank two vector bundle to a line bundle.

Let $\alpha=(3,d)\in\Z^+$ and $\bm{\alpha}\in HN(\alpha)$. The different cases are as follows.
\begin{enumerate}
 \item $s=1$, $\gcd(3,d)=1$, $(\bm{\alpha},\bm{\lambda})=((\alpha),(1))$,
 \item $s=1$, $d=3e$ for some $e\in\Z$, $(\bm{\alpha},\bm{\lambda})=((\alpha),(3))$,
 \item $s=1$, $d=3e$ for some $e\in\Z$, $(\bm{\alpha},\bm{\lambda})=((\alpha),(2,1))$,
 \item $s=1$, $d=3e$ for some $e\in\Z$, $(\bm{\alpha},\bm{\lambda})=((\alpha),(1,1,1))$,
 \item $s=2$, $(\bm{\alpha},\bm{\lambda})=(((1,d_1),(2,d_2)),((1),(1)))$ with $d_2$ odd and $2d_1>d_2$,
 \item $s=2$, $(\bm{\alpha},\bm{\lambda})=(((2,d_1),(1,d_2)),((1),(1)))$ with $d_1$ odd and $d_1>2d_2$,
 \item $s=2$, $(\bm{\alpha},\bm{\lambda})=(((1,d_1),(2,d_2)),((1),(2)))$ with $d_2$ even and $2d_1>d_2$,
 \item $s=2$, $(\bm{\alpha},\bm{\lambda})=(((1,d_1),(2,d_2)),((1),(1,1)))$ with $d_2$ even and $2d_1>d_2$,
 \item $s=2$, $(\bm{\alpha},\bm{\lambda})=(((2,d_1),(1,d_2)),((2),(1)))$ with $d_1$ even and $d_1>2d_2$,
 \item $s=2$, $(\bm{\alpha},\bm{\lambda})=(((2,d_1),(1,d_2)),((1,1),(1)))$ with $d_1$ even and $d_1>2d_2$,
 \item $s=3$, $(\bm{\alpha},\bm{\lambda})=(((1,d_1),(1,d_2),(1,d_3)),((1),(1),(1)))$ with $d_1>d_2>d_3$, $d_1+d_2+d_3=d$
\end{enumerate}

The corresponding Jordan types are as follows.
\begin{enumerate}
 \item $(\alpha)$,
 \item $((0,0),(0,0),(1,e))$,
 \item $((1,e),(1,e))$,
 \item $(\alpha)$,
 \item $((1,d_2-d_1),(1,d_1))$,
 \item $((1,d_1-d_2),(1,d_2))$,
 \item $((0,d_1-\frac{d_2}{2}),(0,0),(1,\frac{d_2}{2}))$,
 \item $((1,d_2-d_1),(1,d_1))$,
 \item $((0,0),(0,\frac{d_1}{2}-d_2),(1,d_2))$,
 \item $((1,d_1-d_2),(1,d_2))$,
 \item $((0,d_1-d_2),(0,d_2-d_3),(1,d_3))$,
\end{enumerate}
\begin{proof}
 The only nontrivial cases are $(7)$ and $(9)$. We only prove $(7)$, the proof of $(9)$ being similar. Let $\mathcal{L}$ and $\mathcal{L}'$ be line bundles of respective degrees $\frac{d_2}{2}$ and $d_1$. Let $\mathcal{F}$ be the nontrivial extension of $\mathcal{L}$ by itself. We therefore have a non-split exact sequence:
 \begin{equation}
 \label{exactseqLF}
  0\rightarrow\mathcal{L}\rightarrow\mathcal{F}\rightarrow\mathcal{L}\rightarrow 0
 \end{equation}

 We have to prove that the Jordan type of a general nilpotent Higgs field for $\mathcal{L}'\oplus\mathcal{F}$ is $((0,d_1-\frac{d_2}{2}),(0,0),(1,\frac{d_2}{2}))$. Such a Higgs field has the form
 \[
  \theta=\begin{pmatrix}
   0&g\\
   0&h
  \end{pmatrix}
 \]
where $g\in\Hom(\mathcal{F},\mathcal{L}')$ and $h\in\End(\mathcal{F})$ is nilpotent. Therefore, $\im h=\mathcal{L}\subset\mathcal{F}$. If $g$ is general, $\theta_{|\mathcal{F}}:\mathcal{F}\rightarrow \mathcal{L}'\oplus\mathcal{F}$ is injective and the restriction $g_{|\mathcal{L}}:\mathcal{L}\rightarrow\mathcal{L}'$ is injective. Indeed, for slope reasons, there exists nonzero morphisms $\mathcal{L}\rightarrow \mathcal{L}'$, which are necessary injective since $\mathcal{L}$ is a line bundle and such a morphism extends to a morphism $\mathcal{F}\rightarrow\mathcal{L}'$ thanks to the exact sequence
\[
 0\rightarrow \Hom(\mathcal{L},\mathcal{L}')\rightarrow \Hom(\mathcal{F},\mathcal{L}')\rightarrow\Hom(\mathcal{L},\mathcal{L}')\rightarrow 0
\]
obtained by applying the left exact functor $\Hom(-,\mathcal{L}')$ to \eqref{exactseqLF} and because $\Ext^1(\mathcal{L},\mathcal{L}')=0$ for slope reasons. Therefore, $[\im \theta]=[\mathcal{F}]=(2,d_2)$, $[\im\theta^2]=(1,\frac{d_2}{2})$ and $\theta^3=0$. This proves that the Jordan type of $\theta$ is $((0,d_1-\frac{d_2}{2}),(0,0),(1,\frac{d_2}{2}))$.
\end{proof}

We observe that it seems not obvious to extract a general pattern or to read the Harder-Narasimhan type of an irreducible component on its Jordan type and vice-versa. Nevertheless, there should be a piecewise linear map which associates a Harder-Narasimhan type to a Jordan-type, that is a map
\[
 (\Z^+)^{(\N_{\geq 1})}\rightarrow (\Z^+)^{(\N_{\geq 1})}
\]
where $(\Z^+)^{(\N)}$ is the set of finitely supported functions $\N_{\geq 1}\rightarrow \Z^+$. This leads to an interesting combinatorial problem which we leave for further investigations. If we see the set of irreducible components of the global nilpotent cone as a basis for the top cohomological Hall algebra of a curve (\cite{2018arXivss}), this question leads to the consideration of analogues of crystals for curves and their parametrizations. For quivers, crystals have been studied for a long time by Kashiwara, Lusztig and now through cluster algebras. The question of reading the Harder-Narasimhan type of an irreducible component of $\mathscr{N}_{\alpha}$ on its Jordan type can be posed for curves of any genus.

\section{Spherical Eisenstein perverse sheaves on the moduli stack of coherent sheaves}
\label{eisensteinsheaves}

\subsection{Induction and restriction functors}
We recall the definitions of the induction and restriction functors appearing in the work of Schiffmann \cite{MR2942792}.
Let $\alpha, \beta\in\Z^+$. We let $\mathfrak{Exact}_{\alpha,\beta}$ be the Artin stack parametrizing exact sequences of coherent sheaves $0\rightarrow \mathcal{G}\rightarrow \mathcal{F}\rightarrow \mathcal{H}\rightarrow 0$ such that $[\mathcal{G}]=\alpha$ and $[\mathcal{H}]=\beta$. We consider the correspondence:
\begin{equation}
\label{inddiagram}
 \xymatrix{
 &\mathfrak{Exact}_{\alpha,\beta}\ar[rd]^{p}\ar[ld]_{q}&\\
 \mathfrak{Coh}_{\beta}(X)\times\mathfrak{Coh}_{\alpha}(X)&&\mathfrak{Coh}_{\alpha+\beta}(X)
 }
\end{equation}
where the morphisms $p$ and $q$ are described on geometric points as follows. If $\mathcal{E}=(0\rightarrow \mathcal{G}\rightarrow \mathcal{F}\rightarrow \mathcal{H}\rightarrow 0)$ is an exact sequence with $[\mathcal{G}]=\alpha$ and $[\mathcal{H}]=\beta$, $p(\mathcal{E})=\mathcal{F}$ and $q(\mathcal{E})=(\mathcal{H},\mathcal{G})$. The morphism $q$ is smooth with connected fibers (and in fact is a vector bundle stack, see \cite{MR3293805,MR2139694}) while $p$ is proper. We let
\[
 \begin{matrix}
  \Ind_{\beta,\alpha}&:&D^b(\mathfrak{Coh}_{\beta}(X))\times D^b(\mathfrak{Coh}_{\alpha}(X))&\rightarrow&D^b(\mathfrak{Coh}_{\alpha+\beta}(X))\\
  && (\mathscr{F},\mathscr{G})&\mapsto&p_!q^*(\mathscr{F}\boxtimes\mathscr{G})[-\langle\beta,\alpha\rangle]
 \end{matrix}
\]
be the induction functor and
\[
 \begin{matrix}
  \Res_{\beta,\alpha}&:&D^b(\mathfrak{Coh}_{\alpha+\beta}(X))&\rightarrow& D^b(\mathfrak{Coh}_{\beta}(X)\times\mathfrak{Coh}_{\alpha}(X))\\
  && \mathscr{F}&\mapsto&q_!p^*\mathscr{F}[-\langle\beta,\alpha\rangle]
 \end{matrix}
\]
be the restriction functor.

\subsection{A category of perverse sheaves on the stack of coherent sheaves}
\label{catsheaf}
In the paper \cite{MR2942792}, Schiffmann considers the semisimple category of perverse sheaves on $\mathfrak{Coh}(X)$ whose simple objects are the simple perverse sheaves on $\mathfrak{Coh}(X)$ which appear with a possible shift as a direct summand of an iterated induction of the constant sheaf on $\mathfrak{Coh}_{\alpha}(X)$ for various $\alpha=(r,d)$ with $r\leq 1$. We let $\mathcal{P}$ be the corresponding category of perverse sheaves on $\mathfrak{Coh}(X)$. These are called \emph{spherical Eisenstein perverse sheaves}. It decomposes as a direct sum $\mathcal{P}=\bigoplus_{\alpha\in\Z^+}\mathcal{P}^{\alpha}$ according to the decomposition of $\mathfrak{Coh}(X)$ into connected components. In his paper, Schiffmann proves that the induction and restriction functors preserve $\mathcal{P}$. Moreover, he gives an explicit description of the simple objects of $\mathcal{P}$ which we recall now. Recall the support map:
\[
 \chi : \mathfrak{Tor}_d(X)\rightarrow S^dX.
\]
It induces an isomorphism
\[
 \mathfrak{Tor}_d^{rss}(X):=\chi^{-1}(S^dX\setminus\Delta)\rightarrow (S^dX\setminus\Delta)/(\G_m)^d
\]
where $(\G_m)^d$ acts trivially on $S^dX\setminus\Delta$. In particular, we have a $\mathfrak{S}_d$-cover $X^d\setminus\Delta\rightarrow S^dX\setminus\Delta$ which gives a family $(\mathscr{L}_{\lambda})_{\lambda\in\mathscr{P}_d}$ of irreducible local systems on $\mathfrak{Tor}_d^{rss}(X)$ indexed by irreducible representations of $\mathfrak{S}_d$, and hence partitions of $d$. This bijection between partitions and irreducible representations of $\mathfrak{S}_d$ is induced by the Springer correspondence, so that the partition $\lambda=(1^d)$ corresponds to the trivial character of $\mathfrak{S}_d$ and $\lambda=(d)$ corresponds to the sign character. Therefore, $\mathscr{L}_{(1^d)}$ is the trivial local system of rank $1$.

Let $\alpha\in\Z^+$ and $\bm{\alpha}=(\alpha_1,\hdots,\alpha_s)\in HN(\alpha)$. We consider the iterated induction diagram:
\begin{equation}
\label{inddiagramit}
 \xymatrix{
 \prod_{i=1}^s\mathfrak{Coh}_{\alpha_{s+1-i}}(X)&\mathfrak{Exact}_{\alpha_1,\hdots,\alpha_s}\ar[r]^(.55){p}\ar[l]_(.4){q}
 &\mathfrak{Coh}_{\alpha}(X)
 }
\end{equation}
where $\mathfrak{Exact}_{\alpha_1,\hdots,\alpha_s}$ parametrizes filtrations $(0=\mathcal{F}_0\subset\hdots\subset\mathcal{F}_s)$ of coherent sheaves such that $[\mathcal{F}_i/\mathcal{F}_{i-1}]=\alpha_i$ for any $1\leq i\leq s$. In the diagram \ref{inddiagramit}, $p^{-1}(\mathfrak{Coh}_{\bm{\alpha}}(X))=q^{-1}\left(\prod_{i=1}^s\mathfrak{Coh}_{(\alpha_{s+1-i})}(X)\right)$. We let $\mathfrak{Exact}_{\bm{\alpha}}$ be this locally closed substack of $\mathfrak{Exact}_{\alpha_1,\hdots,\alpha_s}$. The iterated induction diagram restricted to sheaves of HN-type $\bm{\alpha}$ is:
\begin{equation}
\label{iteratedind}
 \xymatrix{
 &\mathfrak{Exact}_{\bm{\alpha}}\ar[rd]^p\ar[ld]_q&\\
 \prod_{i=1}^s\mathfrak{Coh}_{(\alpha_{s+1-i})}(X)&&\mathfrak{Coh}_{\bm{\alpha}}(X)
 }
\end{equation}
By uniqueness of the Harder-Narasimhan filtration, the map $p$ is an isomorphism. Moreover, $q$ is smooth with connected fibers of dimension $\dim q=d_{\alpha}$ given by Formula \eqref{dimfiberp}. If $\mathscr{F}$ is a perverse sheaf on $\mathfrak{Coh}_{\alpha}(X)$, its supporting stratum is the Harder-Narasimhan stratum $S$ of $\mathfrak{Coh}_{\alpha}(X)$ such that $\supp\mathscr{F}=\overline{\supp\mathscr{F}\cap S}$. We let $\mathfrak{Coh}_{\geq \bm{\alpha}}(X)$ be the union of the Harder-Narasimhan strata $\mathfrak{Coh}_{\bm{\beta}}(X)$ ($\bm{\beta}\in HN(\alpha)$) such that $\mathfrak{Coh}_{\bm{\alpha}}(X)\subset\overline{\mathfrak{Coh}_{\bm{\beta}}(X)}$. It is an open substack of $\mathfrak{Coh}_{\alpha}(X)$. We denote by $j_{\bm{\alpha}}$ the open inclusion $\mathfrak{Coh}_{\geq\bm{\alpha}}(X)\rightarrow\mathfrak{Coh}_{\alpha}(X)$. We let $\mathcal{P}^{\bm{\alpha}}$ denote the restriction to $\mathfrak{Coh}_{\geq \bm{\alpha}}(X)$ of perverse sheaves in $\mathcal{P}^{\alpha}$ whose supporting stratum is $\mathfrak{Coh}_{\bm{\alpha}}(X)$:
\[
 \mathcal{P}^{\bm{\alpha}}=\{j_{\bm{\alpha}}^*\mathscr{F} : \mathscr{F}\in\Ob(\mathcal{P}^{\alpha}), \supp\mathscr{F}=\overline{\supp{\mathscr{F}}\cap\mathfrak{Coh}_{\bm{\alpha}}(X)}\}.
\]

\begin{theorem}[{\cite[Proposition 3.4]{MR2942792}}]
\label{schiffmanneisenstein}
\begin{enumerate}
 \item Let $d\in\N$. The simple objects of the category $\mathcal{P}^{(0,d)}$ are the perverse sheaves on $\mathfrak{Tor}_d(X)$ isomorphic to one of the intersection complexes $\ICC(\mathscr{L}_{\lambda})$ for $\lambda\in\mathscr{P}_d$.
 \item Let $\alpha\in\Z^+$. The simple objects of the category $\mathcal{P}^{\alpha}$ whose supports intersect the semistable locus $\mathfrak{Coh}_{(\alpha)}(X)$ are the perverse sheaves isomorphic to one of the intermediate extensions
 \[
  (j_{\alpha})_{!*}(\epsilon_{\alpha})_*\ICC(\mathscr{L}_{\lambda})
 \]
for $\lambda\in\mathscr{P}_d$, where $j_{\alpha}:\mathfrak{Coh}_{(\alpha)}(X)\rightarrow\mathfrak{Coh}_{\alpha}(X)$ is the inclusion of the semistable locus. For $\alpha\in\Z^+$, we let $\ICC(\mathscr{L}_{\lambda})=(\epsilon_{\alpha})_*\ICC(\mathscr{L}_{\lambda})$. The context makes clear on which space we consider the perverse sheaves.
\item Let $\bm{\alpha}\in HN(\alpha)$. Then, simple perverse sheaves on $\mathfrak{Coh}_{\alpha}(X)$ whose supporting stratum is $\mathfrak{Coh}_{\bm{\alpha}}(X)$ are the perverse sheaves isomorphic to one of the following intermediate extensions
\[
 (j_{\bm{\alpha}})_{!*}p_{\bm{\alpha}}^*(\ICC(\mathscr{L}_{\lambda_1})\boxtimes\hdots\boxtimes\ICC(\mathscr{L}_{\lambda_s}))[d_{\bm{\alpha}}]
\]
for multipartitions $\bm{\lambda}=(\lambda_1,\hdots,\lambda_s)\in\mathscr{P}_{\bm{\delta}}$, where $d_{\bm{\alpha}}$ is the relative dimension of $p_{\bm{\alpha}}$ given by Formula \eqref{dimfiberp}. We let $\mathscr{F}_{\bm{\alpha},\bm{\lambda}}$ denote this perverse sheaf.
 \end{enumerate}
 
\end{theorem}
We order the isomorphism classes of simple spherical Eisenstein perverse sheaves as follows:
\[
 [\mathscr{F}_{\bm{\alpha},\bm{\lambda}}]\leq [\mathscr{F}_{\bm{\beta},\bm{\nu}}]\iff\left\{\begin{aligned}
&\mathfrak{Coh}_{\bm{\beta}}(X)\subset \overline{\mathfrak{Coh}_{\bm{\alpha}}}(X)\text{ is a strict inclusion}\\
&\text{or}\\
&\bm{\beta}=\bm{\alpha} \text{ and for any $1\leq i\leq s$, }\nu_i\leq \lambda_i. \end{aligned}\right.
\]
We define a completion of the Grothendieck group $K_0(\mathcal{P^{\alpha}})$, $\widehat{K_0(\mathcal{P^{\alpha}})}$. By definition, $\widehat{K_0(\mathcal{P^{\alpha}})}$ consists of all formal sums $\sum_{\substack{\bm{\alpha}\in HN(\alpha)\\ \bm{\lambda}\in\mathscr{P}_{\delta}}}a_{\bm{\alpha},\bm{\lambda}}[\mathscr{F}_{\bm{\alpha},\bm{\lambda}}]$ with $a_{\bm{\alpha},\bm{\lambda}}\in\Z$.

\subsection{Singular support of spherical Eisenstein perverse sheaves}
\subsubsection{Nilpotency}
Consider $\alpha\in\Z^+$, $s\geq 0$ and $\bm{\alpha}=(\alpha_1,\hdots,\alpha_s)\in(\Z^+)^s$ such that $\sum_{i=1}^s\alpha_i=\alpha$. We let $\mathfrak{Y}_{\alpha}=\mathfrak{Coh}_{\alpha}(X)$, $\mathfrak{Y}_{\bm{\alpha}}=\prod_{i=1}^s\mathfrak{Coh}_{\alpha_i}(X)$ and $\mathfrak{X}_{\bm{\alpha}}$ be the stack parametrizing filtrations $(0=\mathcal{F}_0\subset \mathcal{F}_1\subset \hdots\subset \mathcal{F}_s=\mathcal{F})$ such that for $1\leq i\leq s$, $[\mathcal{F}_i/\mathcal{F}_{i-1}]=\alpha_i$. The iterated induction diagram is
\[
 \xymatrix{
 \mathfrak{Y}_{\bm{\alpha}}&\mathfrak{X}_{\bm{\alpha}}\ar[r]^p\ar[l]_q&\mathfrak{Y}_{\alpha}
 }
\]

\begin{proposition}
\label{singsuppnil}
 Spherical Eisenstein perverse sheaves have nilpotent singular support.
\end{proposition}
\begin{proof}
 It suffices to show that $p_!q^*\underline{\C}_{\mathfrak{Y}_{\bm{\alpha}}}$ has a nilpotent singular support for any $\bm{\alpha}\in (\Z^+)^s$. Since $q^*\underline{\C}_{\mathfrak{Y}_{\bm{\alpha}}}$ is the constant local system on $\mathfrak{X}_{\bm{\alpha}}$, it suffices to understand how the zero section of $T^*\mathfrak{X}_{\bm{\alpha}}$ transforms under the cotangent correspondence:
 \[
  \xymatrix{
  &T^*\mathfrak{Y}_{\alpha}\times_{\mathfrak{Y}_{\alpha}}\mathfrak{X}_{\bm{\alpha}}\ar[ld]_{dp^*}\ar[rd]^{pr_1}&\\
  T^*\mathfrak{X}_{\bm{\alpha}}&&T^*\mathfrak{Y}_{\alpha}
  }
 \]
We can describe $T^*\mathfrak{Y}_{\alpha}$ as the stack of Higgs bundles, which parametrizes pairs $(\mathcal{F},\theta)$, where $\mathcal{F}$ is a coherent sheaf on $X$ verifying $[\mathcal{F}]=\alpha$ and $\theta\in\Hom_{\OO_X}(\mathcal{F},\mathcal{F}\otimes K_X)$ (see Section \ref{higgssheaves}); analogously, $T^*\mathcal{Y}_{\bm{\alpha}}=\prod_{i=1}^sT^*\mathfrak{Y}_{\alpha_i}$ and $T^*\mathfrak{Y}_{\alpha}\times_{\mathfrak{Y}_{\alpha}}\mathfrak{X}_{\bm{\alpha}}$ is the stack parametrizing pairs $(\mathcal{F}_{\bullet},\theta)$ where $\mathcal{F}_{\bullet}=(0=\mathcal{F}_0\subset\hdots\subset \mathcal{F}_s=\mathcal{F})$ and $\theta\in\Hom_{\OO_X}(\mathcal{F},\mathcal{F}\otimes K_X)$. The map $pr_1$ sends $(\mathcal{F}_{\bullet},\theta)$ to $(\mathcal{F},\theta)$. If we write $\mathfrak{X}_{\bm{\alpha}}$ for the zero-section of $T^*\mathfrak{X}_{\bm{\alpha}}$, the geometric points of $(dp^*)^{-1}(\mathfrak{X}_{\bm{\alpha}})$ parametrize pairs $(\mathcal{F}_{\bullet},\theta)$ such that $\theta(\mathcal{F}_i)\subset \mathcal{F}_{i-1}\otimes K_X$ for any $1\leq i\leq s$. Consequently, $pr_1(dp^*)^{-1}(\mathfrak{X}_{\bm{\alpha}})\subset \mathscr{N}_{\alpha}$ and since the properness of $p$ implies $SS(p_!q^*\underline{\C}_{\mathfrak{Y}_{\bm{\alpha}}})\subset pr_1(dp^*)^{-1}(\mathfrak{X}_{\bm{\alpha}})$, this proves the nilpotency of the singular support.
\end{proof}

\begin{remark}
 The same proof shows that for any local system $\mathscr{L}$ on $\mathfrak{X}_{\bm{\alpha}}$, the singular support of $p_!\mathscr{L}$ is nilpotent.
\end{remark}

\begin{remark}
 The situation of Proposition \ref{singsuppnil} is analogous to the situation of the induction of perverse sheaves for reductive Lie algebras. Let $P\subset G$ be a parabolic subgroup of a reductive algebraic group with L\'evi quotient $L$ and $\mathfrak{p}\subset \mathfrak{g}$ and $\mathfrak{l}$ be the corresponding Lie algebras. In this situation, the induction diagram reads
 \[
  \xymatrix{\mathfrak{l}/L&\ar[l]_q\mathfrak{p}/P\ar[r]^p&\mathfrak{g}/G}
 \]
(see for example \cite[\S 1.1]{MR3874694}).

\end{remark}

\subsubsection{The singular support over the supporting HN-stratum}
\label{singsuppsup}
In this section, we compute the singular support of the restriction of spherical Eisenstein sheaves to their supporting stratum. Let $\mathscr{F}$ be a simple spherical Eisenstein sheaf on $\mathfrak{Coh}_{\alpha}(X)$ supported on the stratum $\mathfrak{Coh}_{\bm{\alpha}}(X)$ (that is, $\supp\mathscr{F}=\overline{\supp\mathscr{F}\cap\mathfrak{Coh}_{\bm{\alpha}}(X)}$). Recall the open immersion $j_{\bm{\alpha}} : \mathfrak{Coh}_{\geq\bm{\alpha}}(X)\rightarrow \mathfrak{Coh}_{\alpha}(X)$ from Section \ref{catsheaf}. The characteristic cycle of $(j_{\bm{\alpha}})^*\mathscr{F}$ is a $\Z$-linear combination with nonnegative coefficients of the irreducible components of $\overline{\mathscr{N}_{\bm{\alpha}}}$. By Theorem \ref{mainirr}, we can write
\begin{equation}
\label{deccc}
 CC((j_{\bm{\alpha}})^*\mathscr{F})=\sum_{\bm{\lambda}\in\mathscr{P}_{\bm{\delta}}}m_{\mathscr{F},\bm{\lambda}}[\overline{\mathscr{N}_{\bm{\alpha},\bm{\lambda}}}].
\end{equation}
See Section \ref{ircompnc} for details on the notation $\mathscr{N}_{\bm{\alpha},\bm{\lambda}}$. Write $\mathscr{F}=(j_{\bm{\alpha}})_{!*}p^*_{\bm{\alpha}}(\ICC(\mathscr{L}_{\lambda_1})\boxtimes\hdots\boxtimes \ICC(\mathscr{L}_{\lambda_s}))[d_{\bm{\alpha}}]$ (using Theorem \ref{schiffmanneisenstein}). Then, the restriction of $\mathscr{F}$ to $\mathfrak{Coh}_{\bm{\alpha}}(X)$ is $p_{\bm{\alpha}}^*(\ICC(\mathscr{L}_{\lambda_1})\boxtimes\hdots\boxtimes \ICC(\mathscr{L}_{\lambda_s}))[d_{\bm{\alpha}}]$. We let $\mathscr{F}_i=\ICC(\mathscr{L}_{\lambda_i})$ for $1\leq i\leq s$. This is a perverse sheaf on $\mathfrak{Coh}_{\alpha_i}(X)$. With these notations, we have in particular
\[
 CC((j_{(\alpha_i)})^*\mathscr{F}_i)=\sum_{\lambda\in\mathscr{P}_{\delta_i}}m_{\mathscr{F}_i,\lambda}[\overline{\mathscr{N}_{(\alpha_i),\lambda}}].
\]
This is Formula \eqref{deccc} applied to $\alpha=\alpha_i$ and $\bm{\alpha}=(\alpha_i)$, using that $d_{(\alpha_i)}=0$.

\begin{lemma}
\label{multss}
 The multiplicies $m_{\mathscr{F},\bm{\lambda}}$ are given by the formula
 \[
  m_{\mathscr{F},\bm{\lambda}}=\prod_{i=1}^sm_{\mathscr{F}_i,\lambda_i}
 \]
for any $\bm{\lambda}=(\lambda_1,\hdots,\lambda_s)\in\mathscr{P}_{\bm{\delta}}$.
\end{lemma}
\begin{proof}
 This formula follows from the fact that the characteristic cycle of an exterior product is the product of the characteristic cycles, from the smoothness of $q$ and its compatibility with the stratifications defined in Section \ref{refhnstrat} (see \cite{MR1074006}). More precisely, Proposition $9.4.5$ of \emph{op. cit.} asserts that a shift by $k$ transforms the characteristic cycle by $(-1)^k$, Proposition $9.4.3$ that the characteristic cycle map is compatible with pull-back by non-characteristic morphisms and the remark following Definition $5.4.12$ that a smooth morphism in non-characteristic.
\end{proof}

Before giving the microlocal multiplicities of spherical Eisenstein sheaves whose supporting stratum is the semistable one, that is $\bm{\alpha}=(\alpha)$, we need a digression on perverse sheaves on general linear Lie algebras $\mathfrak{gl}_d$ for $d\geq 0$. Let $d\in\N$. Nilpotent orbits of $\mathfrak{gl}_d$ are in bijection with partitions of $d$ in such a way that for any two partitions $\lambda,\nu$, $\OO_{\nu}\subset \overline{\OO_{\lambda}}$ if and only if $\nu\leq \lambda$ for the dominance order $\leq$ on partitions. In particular, the orbit $\{0\}\subset\mathfrak{gl}_d$ corresponds to the partition $(1^d)$ and the regular nilpotent orbit to $\lambda=(d)$.

Let $\Lambda_{\nu}=[\overline{T^*_{\OO_{\nu}}\mathfrak{gl}_d}]$ as a Lagrangian cycle of $T^*\mathfrak{gl}_d$. Then, we have the following easy lemma.

\begin{lemma}
\label{ccisonil}
 The characteristic cycle map
 \[
  CC : K_0(\Perv_{\GL_d}(\mathscr{N}))\rightarrow \Z[\Lambda_{\nu} : \nu\in\mathscr{P}_d]
 \]
 is an isomorphism of $\Z$-modules. Moreover, it is lower unitriangular with respect to the basis of simple perverse sheaves on the left and the basis given by the $\Lambda_{\nu}$ on the right, both ordered using the antidominance order on partitions.

\end{lemma}
Consider the Grothendieck-Springer resolution for $\mathfrak{g}=\mathfrak{gl}_d$:
\[
 \pi_{\mathfrak{g}} : \tilde{\mathfrak{g}}\rightarrow \mathfrak{g}.
\]
The decomposition theorem gives the decomposition $(\pi_{\mathfrak{g}})_*\underline{\C}=\bigoplus_{\lambda\in\mathscr{P}_d}\ICC(\mathscr{L}_{\lambda})\otimes V_{\lambda}$ where $\ICC(\mathscr{L}_{\lambda})=\mathfrak{F}\ICC(\OO_{\lambda})$ and $V_{\lambda}$ is a (non-zero) multiplicity complex. Here, $\mathfrak{F}$ is the Fourier-Sato transform of perverse sheaves on $\mathfrak{g}$ and $\mathscr{L}_{\lambda}$ is the local system on $\mathfrak{g}$ associated to the partition $\lambda$ (the map $\pi_{\mathfrak{g}}$ is small and is a $\mathfrak{S}_d$-cover over the regular semisimple locus of $\mathfrak{g}$). In particular, $\mathscr{L}_{(1^d)}=\underline{\C}_{\mathfrak{g}}$ is the trivial local system while $\mathscr{L}_{(d)}$ is associated to the sign character of $\mathfrak{S}_d$. Precomposing the isomorphism of Lemma \ref{ccisonil} with the Fourier-Sato transform gives a $\Z$-module isomorphism
\begin{equation}
\label{ccnilconeg}
 CC : K_0(\mathcal{P}_{\mathfrak{g}})\rightarrow \Z[\Lambda_{\nu} : \nu\in\mathscr{P}_d],
\end{equation}
where $\mathcal{P}_{\mathfrak{g}}$ is the semisimple category of perverse sheaves on $\mathfrak{g}$ generated by the $\ICC(\mathscr{L}_{\lambda})$, $\lambda\in\mathscr{P}_d$. Moreover, by ordering the basis given by classes of simple perverse sheaves on the left and the basis $(\Lambda_{\nu})_{\nu\in\mathscr{P}_d}$ on the right by the antidominance order, this isomorphism is lower unitriangular. Note that identifying $T^*\mathfrak{g}$, $T^*\mathfrak{g}^*$ and $\mathfrak{g}\times \mathfrak{g}$ in the natural way using the trace pairing, the Fourier transform and the characteristic cycle map are compatible with these isomorphisms, that is the microlocal multiplicities are preserved (see \cite[Exercise IX.7]{MR1074006}). It is possible to give explicitly the microlocal multiplicities thanks to the following result of Evens and Mirkovi\'c.
\begin{theorem}[{\cite[Theorem 0.2 (b)]{MR1682280}}]
 Let $d\in\N$. Let $\lambda$ and $\nu$ be two partitions of $d$ corresponding to nilpotent orbits $\OO_{\lambda}$ and $\OO_{\nu}$. The multiplicity $\alpha_{\nu,\lambda}$ of $[\overline{T^*_{\OO_{\nu}}\mathfrak{gl}_d}]$ in $CC(\ICC(\OO_{\lambda}))$ is given by the multiplicity of the Springer representation $V_{\lambda}$ of $\mathfrak{S}_d$ in the cohomology of the Springer fiber $H^*(\mathcal{B}_e)$ at $e\in\OO_{\nu}$, which is the Kostka number $K_{\lambda\nu}$.
\end{theorem}

Let $\mathscr{F}$ be a simple spherical Eisenstein sheaf on $\mathfrak{Coh}_{\alpha}(X)$ whose support intersects the semistable stratum $\mathfrak{Coh}_{(\alpha)}(X)$ (that is, the restriction $(j_{(\alpha)})^*\mathscr{F}$ is a nonzero perverse sheaf of the category $\mathcal{P}^{(\alpha)}$). By Theorem \ref{schiffmanneisenstein}, there exists a partition $\lambda\in\mathscr{P}$ such that $\mathscr{F}=\ICC(\mathscr{L}_{\lambda})$. Then, the characteristic cycle of $(j_{(\alpha)})^*\mathscr{F}$ is given by the following lemma.

\begin{lemma}
\label{sstablemult}
 We have
 \[
  CC((j_{(\alpha)})^*\mathscr{F})=\sum_{\nu\leq \lambda}m_{\nu,\lambda}[\overline{\mathscr{N}_{(\alpha),\nu}}],
 \]
where $m_{\nu,\lambda}=\alpha_{\nu,\lambda}$.
In particular, the map
\[
\begin{matrix}
 CC &:&K_0(\mathcal{P}^{(\alpha)})&\rightarrow& \Z[\overline{\mathscr{N}_{(\alpha),\nu}}:\nu\in\mathscr{P}_{\delta}]\\
 &&\mathscr{F}&\mapsto&CC(\mathscr{F})

 \end{matrix}
\]
is a lower unitriangular isomorphism of $\Z$-modules when the basis of simple perverse sheaves on the left and irreducible components of $\mathscr{N}_{(\alpha)}$ on the right are ordered using the antidominance order on partitions.
\end{lemma}
\begin{proof}
 We first reduce to the case $\alpha=(0,d)$ for some $d\geq 0$. Recall the isomorphism of stacks
 \[
  \epsilon_{\alpha} : \mathfrak{Coh}_{(0,\delta)}(X)\rightarrow \mathfrak{Coh}_{(\alpha)}(X)
 \]
where $\delta=\gcd(\alpha)$. It is induced by an equivalence of categories so it also gives an isomorphism at the level of the stacks of Higgs bundles:
\[
 \epsilon_{\alpha}:\mathfrak{Higgs}_{(0,\delta)}(X)\rightarrow \mathfrak{Higgs}_{(\alpha)}(X)
\]
making the following natural diagram commute:
\[
 \xymatrix{
 \mathfrak{Higgs}_{(0,\delta)}(X)\ar[r]^{\epsilon_{\alpha}}\ar[d]_{\pi_{(0,\delta)}}&\mathfrak{Higgs}_{(\alpha)}(X)\ar[d]^{\pi_{\alpha}}\\
 \mathfrak{Coh}_{(0,\delta)}(X)\ar[r]^{\epsilon_{\alpha}}&\mathfrak{Coh}_{(\alpha)}(X)
 }
\]
and $\epsilon_{\alpha}$ also induces an isomorphism of the semistable parts of the global nilpotent cones:
\[
 \epsilon_{\alpha} : \mathscr{N}_{(0,\delta)}\rightarrow \mathscr{N}_{(\alpha)}.
\]
Therefore, we can assume $\alpha=(0,d)$ for some $d\geq 0$. The problem is then local (and does not depend anymore on $X$ being an elliptic curve), so that we can assume $X=\A^1$ (for the same reason as in the proof of \cite[Th\'eor\`eme (3.3.13)]{MR899400}). We are now in the situation of the classical Springer correspondence for $\mathfrak{gl}_d$ and the Theorem is a consequence of the properties of the map \eqref{ccnilconeg}.
\end{proof}

\begin{lemma}
\label{unitriang}
 Let $\alpha\in\Z^+$ and $\bm{\alpha}\in HN(\alpha)$. The characteristic cycle map
 \[
  \begin{matrix}
   CC&:&K_0(\mathcal{P}^{\bm{\alpha}})&\mapsto&\Z[\overline{\mathscr{N}_{\bm{\alpha},\bm{\lambda}}}:\bm{\lambda}\in\mathscr{P}_{\bm{\delta}}]\\
   &&\mathscr{F}&\mapsto&CC(\mathscr{F})
  \end{matrix}
 \]
is a lower unitriangular isomorphism of $\Z$-modules when the set of $s$-uplets of partitions $\mathscr{P}_{\bm{\delta}}$ is ordered by the antidominance order, that is
\[
 \bm{\lambda}\leq\bm{\nu}\iff \text{for $1\leq i\leq s$, }\lambda_i\geq \nu_i.
\]
\end{lemma}
\begin{proof}
 It is an immediate consequence of Lemmas \ref{sstablemult} and \ref{multss}.
\end{proof}

\section{Proofs of the main theorems \ref{mainirr} and \ref{mainbij}}
\label{Proofs}

\subsection{Some Lemmas}
\label{proofthm1}
\begin{lemma}
\label{lemmanilend}
 Let $d\geq 1$, $\xi\in(\N^{\mathscr{P}})_d$ and $T=\bigoplus_{i=1}^sT_{x_i,\lambda_i}$ be a torsion sheaf in the stratum $\mathfrak{Coh}_{(0,d),\xi}(X)$. Then, the closed subset $\mathscr{N}_T$ of $\End(T)$ of nilpotent endomorphisms is irreducible of codimension
 \[
  \sum_{\lambda\in\mathscr{P}}\xi(\lambda)l(\lambda).
 \]

\end{lemma}
\begin{proof}
 Note that $\sum_{\lambda\in\mathscr{P}}\xi(\lambda)l(\lambda)$ is the number of indecomposable summands of $T$ (see Section \ref{strattorsion}). Write
 \[
  T=\bigoplus_{j=1}^tT_j^{\oplus m_j}
 \]
where the $T_j$ are the pairwise distinct indecomposable summands of $T$ and $m_j$ are their multiplicities. We have to show that the codimension of $\mathscr{N}_T$ in $\End(T)$ is $\sum_{j=1}^tm_j$. Let $J$ be the radical of $\End(T)$. The quotient $\End(T)/J$ is isomorphic to
\[
 \prod_{j=1}^t\mathscr{M}_{m_i}(\C)
\]
where for any $n\in\N$, $\mathscr{M}_n(\C)$ is the ring of $n\times n$ matrices with complex coefficients. The projection
\[
 p : \End(T)\rightarrow \End(T)/J
\]
is a fiber bundle and $f\in\End(T)$ is nilpotent if and only if $p(f)$ is nilpotent. Therefore,
\[
 \mathscr{N}_T=p^{-1}\left(\prod_{j=1}^t\mathscr{N}_{m_j}\right)
\]
where $\mathscr{N}_{m_j}$ denotes the nilpotent cone of $\mathscr{M}_{m_j}(\C)$, and for any $n\in\N$, the nilpotent cone of $\mathscr{M}_n(\C)$ is irreducible of codimension $n$ (it is the vanishing locus of the $n$ symmetric polynomials in the $n$ eigenvalues). The result for $\mathscr{N}_T$ follows.
 \end{proof}

For $\alpha\in\Z^+$ and $\xi\in(\N^{\mathscr{P}})_{\delta}$, the dimension of the endomorphism ring of a semistable coherent sheaf whose isomorphism class belongs to $\lvert\Coh_{(\alpha),\xi}(X)\rvert$ is constant and only depends on $\xi$. We denote it $e(\xi)$.

Next, we need the dimension of the stratum $\mathfrak{Coh}_{(\alpha),\xi}(X)$. Thanks to the isomorphism $\epsilon_{\alpha}$ between $\mathfrak{Coh}_{(0,\delta)}(X)$ and $\mathfrak{Coh}_{(\alpha)}(X)$ ($\delta=\gcd{\alpha}$), we can assume that $\alpha=(0,\delta)$. Then, $\sum_{\lambda\in\mathscr{P}}\xi(\lambda)$ is the number of parameters of $\mathfrak{Coh}_{(0,\delta),\xi}(X)$ (see Section \ref{torstack}) and $e(\xi)$ is the dimension of the automorphism group. We easily deduce the following lemma.
\begin{lemma}
\label{lemdimstratss}
 Let $\alpha\in\Z^+$ and $\xi\in(\N^{\mathscr{P}})_{\delta}$. Then, $\mathfrak{Coh}_{(\alpha),\xi}(X)$ is irreducible of dimension
 \[
  \sum_{\lambda\in\mathscr{P}}\xi(\lambda)-e(\xi).
 \]

\end{lemma}

\begin{lemma}
\label{lemdimstrat}
 For any $\bm{\xi}\in(\N^{\mathscr{P}})_{\bm{\delta}}$, the stratum $\mathfrak{Coh}_{\bm{\alpha},\bm{\xi}}(X)$ is irreducible of dimension
 \[
  \sum_{i=1}^s\left(\sum_{\lambda\in\mathscr{P}}\xi_i(\lambda)-e(\xi_i)\right)-\sum_{i<j}(r_jd_i-r_id_j)
 \]

\end{lemma}
\begin{proof}
 The morphism $p_{\bm{\alpha}} : \mathfrak{Coh}_{\bm{\alpha}}(X)\rightarrow \prod_{i=1}^s\mathfrak{Coh}_{(\alpha_i)}(X)$ is a stack bundle with fibers of dimension $d_{\bm{\alpha}}$ given by Formula \eqref{dimfiberp}. Together with Lemma \ref{lemdimstratss}, this gives Lemma \ref{lemdimstrat}.
\end{proof}

\begin{lemma}
\label{lemdimfiber}
 For any $\bm{\xi}\in (\N^{\mathscr{P}})_{\bm{\delta}}$, and any geometric point $\mathcal{F}\in\mathfrak{Coh}_{\bm{\alpha},\bm{\xi}}(X)$, $\pi_{\alpha,\mathscr{N}}^{-1}(\mathcal{F})$ is irreducible of dimension
 \[
  \sum_{i<j}(r_jd_i-r_id_j)+\sum_{i=1}^s\left(e(\xi_i)-\sum_{\lambda\in\mathscr{P}}\xi_i(\lambda)l(\lambda)\right).
 \]

\end{lemma}
\begin{proof}
The geometric points of the fiber of $\pi_{\alpha,\mathscr{N}}$ over $\mathcal{F}\in\mathfrak{Coh}_{\bm{\alpha},\bm{\xi}}(X)$ are nilpotent endomorphisms of $\mathcal{F}$. Let $\mathcal{F}=\bigoplus_{i=1}^s\mathcal{F}_i$ be its Harder-Narasimhan decomposition (since the Harder-Narasimhan filtration splits), where $[\mathcal{F}_i]=\alpha_i=(r_i,d_i)$. An endomorphism of $\mathcal{F}$ is the datum of $f_{ij}\in \Hom(\mathcal{F}_i,\mathcal{F}_j)$ for $1\leq j\leq i\leq s$ and it is nilpotent if and only if $f_{i,i}$ is nilpotent for any $1\leq i\leq s$. From the equality
\[
 \dim\Hom(\mathcal{F}_i,\mathcal{F}_j)=r_id_j-r_jd_i
\]
and Lemma \ref{lemmanilend}, we get the desired formula.

\end{proof}

\begin{lemma}
\label{dimcompirr}
 For $\bm{\xi}\in (\N^{\mathscr{P}})_{\bm{\delta}}$, $\pi_{\alpha,\mathscr{N}}^{-1}(\mathfrak{Coh}_{\bm{\alpha},\bm{\xi}}(X))$ is irreducible of dimension
 \[
  \sum_{i=1}^s\sum_{\lambda\in\mathscr{P}}\xi_i(\lambda)(1-l(\lambda)).
 \]
In particular, this dimension is nonpositive and equals $0=\dim\mathscr{N}$ if and only if for any $1\leq i\leq s$ and $\lambda\in\mathscr{P}$, $\xi_i(\lambda)\neq 0$ implies $l(\lambda)=1$, which is the definition of $\bm{\xi}$ being regular.
\end{lemma}
\begin{proof}
 The restriction of $\pi_{\alpha,\mathscr{N}}$
 \[
  \pi_{\alpha,\mathscr{N}} : \pi_{\alpha,\mathscr{N}}^{-1}(\mathfrak{Coh}_{\bm{\alpha},\bm{\xi}}(X))\rightarrow \mathfrak{Coh}_{\bm{\alpha},\bm{\xi}}(X)
 \]
is surjective with irreducible target (whose dimension is given by Lemma \ref{lemdimstrat}) and irreducible fibers (of dimension given by Lemma \ref{lemdimfiber}). Hence, $\pi_{\alpha,\mathscr{N}}^{-1}(\mathfrak{Coh}_{\bm{\alpha},\bm{\xi}}(X))$ is irreducible, of dimension
\[
 \sum_{i=1}^s\left(\sum_{\lambda\in\mathscr{P}}\xi_i(\lambda)-e(\xi_i)\right)-\sum_{i<j}(r_jd_i-r_id_j)+ \sum_{i<j}(r_jd_i-r_id_j)+\sum_{i=1}^s\left(e(\xi_i)-\sum_{\lambda\in\mathscr{P}}\xi_i(\lambda)l(\lambda)\right)
\]
which yields the formula of the Lemma.
\end{proof}

\subsection{Proof of Theorem \ref{mainirr}}
 Let $\Lambda$ be an irreducible component of $\mathscr{N}_{\bm{\alpha}}$. It is of dimension $0$. Since
 \[
  \mathscr{N}_{\bm{\alpha}}=\bigcup_{\bm{\xi}\in (\N^{\mathscr{P}})_{\bm{\delta}}}\pi_{\alpha,\mathscr{N}}^{-1}(\mathfrak{Coh}_{\bm{\alpha},\bm{\xi}}(X))
 \]
and this is a locally closed stratification of $\mathscr{N}_{\bm{\alpha}}$, there exists $\bm{\xi}$ such that $\pi_{\alpha,\mathscr{N}}^{-1}(\mathfrak{Coh}_{\bm{\alpha},\bm{\xi}}(X))\cap\Lambda$ is open and dense in $\Lambda$. Therefore, $\pi_{\alpha,\mathscr{N}}^{-1}(\mathfrak{Coh}_{\bm{\alpha},\bm{\xi}}(X))$ is of dimension $0=\dim\mathscr{N}$ hence $\bm{\xi}$ is regular by Lemma \ref{dimcompirr}. Write $\bm{\xi}=\bm{\xi}_{\bm{\lambda}}$ for some $\bm{\lambda}\in \mathscr{P}_{\bm{\delta}}$. Then, $\Lambda=\overline{\pi_{\alpha,\mathscr{N}}^{-1}(\mathfrak{Coh}_{\bm{\alpha},\bm{\xi}}(X))}=\overline{\mathscr{N}_{\bm{\alpha},\bm{\lambda}}}$.

For the converse, if $\bm{\xi}$ is regular, by Lemma \ref{dimcompirr}, $\pi_{\alpha,\mathscr{N}}^{-1}(\mathfrak{Coh}_{\bm{\alpha},\bm{\xi}}(X))$ is an irreducible substack of $\mathscr{N}_{\bm{\alpha}}$ of dimension $0$ and hence its closure is an irreducible component of $\mathscr{N}_{\bm{\alpha}}$.

This proves the description of the irreducible components of $\mathscr{N}_{\bm{\alpha}}$. The description of the irreducible components of $\mathscr{N}_{\alpha}$ follows immediately. Indeed, if $\Lambda$ is such an irreducible component, we let $\mathfrak{Coh}_{\bm{\alpha}}(X)$ be its supporting stratum. Then, $\Lambda\cap \pi_{\alpha}^{-1}(\mathfrak{Coh}_{\bm{\alpha}}(X))$ is an irreducible component of $\mathscr{N}_{\bm{\alpha}}$. Therefore, it is of the form $\overline{\mathscr{N}_{\bm{\alpha},\bm{\lambda}}}$ for some $\bm{\lambda}\in\mathscr{P}_{\bm{\delta}}$. Hence, $\Lambda=\Overline[1.5]{\mathscr{N}_{\bm{\alpha},\bm{\lambda}}}$.

\subsection{Proof of Theorem \ref{mainbij}}
\label{proofthm}
Let $\mathscr{F}\in \mathcal{P}^{\alpha}$ be a simple Eisenstein perverse sheaf. Let $\mathfrak{Coh}_{\bm{\alpha}}(X)$ be its supporting stratum. Then, its characteristic cycle can be written
\[
 CC(\mathscr{F})=\sum_{(\bm{\beta},\bm{\lambda})}m_{\bm{\beta},\bm{\lambda}}[\Overline[1.5]{\mathscr{N}_{\bm{\beta},\bm{\lambda}}}]
\]
where the sum runs over pairs $(\bm{\beta},\bm{\lambda})$ of a Harder-Narasimhan type $\bm{\beta}=(\beta_1,\hdots,\beta_s)\in HN(\alpha)$ and a multipartition $\bm{\lambda}\in\mathscr{P}_{\bm{\delta}}$, $\bm{\delta}=\gcd(\bm{\beta})$. Moreover, $m_{\bm{\beta},\bm{\lambda}}=0$ unless $\mathfrak{Coh}_{\bm{\beta}}(X)\subset \overline{\mathfrak{Coh}_{\bm{\alpha}}(X)}$ and
\[
 CC(j_{\bm{\alpha}}^*\mathscr{F})=\sum_{\bm{\lambda}\in\mathscr{P}_{\bm{\delta}}}m_{\bm{\alpha},\bm{\lambda}}[\overline{\mathscr{N}_{\bm{\alpha},\bm{\lambda}}}]
\]
is given by Lemma \ref{multss}. By Lemma \ref{unitriang}, we obtain the lower unitriangularity of the characteristic cycle map
\[
 CC : \widehat{K_0(\mathcal{P}^{\alpha})}\rightarrow \widehat{\Z[\Irr(\mathscr{N}_{\alpha})]}.
\]

\section{Perverse sheaves with nilpotent singular support on the stack of coherent sheaves}
\label{nilsingps}
In this section, we will describe explicitly the simple objects of the category $\Perv(\mathfrak{Coh}_{\alpha}(X),\mathscr{N}_{\alpha})$ of perverse sheaves on the stack of coherent sheaves on an elliptic curve whose singular support is nilpotent (that is, a union of some of the irreducible components of $\mathscr{N}_{\alpha}$).

\subsection{Local systems on the semistable locus of the stack of coherent sheaves}

\subsubsection{Local systems on the Picard stack}
\begin{proposition}
\label{locsyspicard}
 Let $\alpha=(1,d)\in\Z^+$ and $\mathscr{L}$ be a local system on $\mathfrak{Pic}_{d}=\mathfrak{Coh}_{(\alpha)}(X)\subset \mathfrak{Coh}_{\alpha}(X)$. Then, $\mathscr{L}$ extends to a local system on $\mathfrak{Coh}_{\alpha}(X)$. 
\end{proposition}

\begin{proof}
Consider the determinant morphism
\[
 \det : \mathfrak{Coh}_{\alpha}(X) \rightarrow \mathfrak{Pic}_d.
\]
It restricts to the identity on $\mathfrak{Coh}_{(\alpha)}(X)=\mathfrak{Pic}_d$. Let $\mathscr{L}$ be a local system on $\mathfrak{Coh}_{(\alpha)}(X)$. Then, $\mathscr{L}=j_{(\alpha)}^*\det^*(\mathscr{L})$. Consequently, $\det^*(\mathscr{L})$ is a local system on $\mathfrak{Coh}_{\alpha}(X)$ which extends $\mathscr{L}$.
\end{proof}

\subsubsection{Codimension one Harder-Narasimhan strata of $\mathfrak{Coh}_{\alpha}(X)$}
\begin{proposition}
\label{codimonestratum}
 Let $\alpha=(r,d)\in\Z^+$. Then, $\mathfrak{Coh}_{\alpha}(X)$ has a codimension one Harder-Narasimhan stratum if and only if $r$ and $d$ are coprime. In this case, such a stratum is unique.
\end{proposition}
\begin{proof}
 Let $\bm{\alpha}=(\alpha_1,\hdots,\alpha_s)\in HN(\alpha)$ be a Harder-Narasimhan type. By Formula \eqref{dimfiberp}, $\dim\mathfrak{Coh}_{\bm{\alpha}}(X)=-\sum_{j<i}(r_id_j-r_jd_i)$ when we write $\alpha_i=(r_i,d_i)$. Consequently, the codimension of $\mathfrak{Coh}_{\bm{\alpha}}(X)$ in $\mathfrak{Coh}_{\alpha}(X)$ is the opposite, since $\dim\mathfrak{Coh}_{\alpha}(X)=0$. Since the slopes strictly decrease, for any $j<i$, $\frac{d_j}{r_j}>\frac{d_i}{r_i}$. Each term of the sum is therefore positive.
 
 Assume that $\mathfrak{Coh}_{\bm{\alpha}}(X)$ is of codimension one. Then necessarily $\bm{\alpha}=(\alpha_1,\alpha_2)$ has length two and $r_2d_1-r_1d_2=1$. Using that $r_2=r-r_1$ and $d_2=d-d_1$, we get $rd_1-r_1d=1$, that is a B\'ezout relation between $r$ and $d$. This proves that $r$ and $d$ have to be coprime. If $rd_1'-r_1'd=1$ is an other B\'ezout relation between $r$ and $d$, there exists $k\in\Z$ such that
 \[
  \left\{\begin{aligned}
          r_1'=r_1+kr\\
          d_1'=d_1+kd
         \end{aligned}
  \right.
 \]
If this new B\'ezout relation comes from a codimension one stratum given by $\bm{\alpha}'=(\alpha_1',\alpha_2')\in HN(\alpha)$, we have furthermore $0\leq r_1'\leq r$. If $0<r_1<r$, then necessarily $r_1'=r_1$ and $d_1'=d_1$ so that $\bm{\alpha}'=\bm{\alpha}$. The case $r_1=r$ is excluded since in this case, $\bm{\alpha}=((r,d_1),(0,d-d_1))$ and the slopes do not decrease. If $r_1=0$, $r_1'=0$ or $r_1'=r$ but the second case is not allowed for the same reason. Hence $\bm{\alpha}=\bm{\alpha}'$ in any case. This proves one implication and the last statement of the Proposition. For the converse, assume $(r,d)$ is coprime. Let $rd_1-r_1d=1$ be the B\'ezout relation with $0\leq r_1<r$. We let $\bm{\alpha}=(\alpha_1,\alpha_2)=((r_1,d_1),(r-r_1,d-d_1))$. Let $r_2=r-r_1$ and $d_2=d-d_1$. The B\'ezout relation can be rewritten $r_2d_1-r_1d_2=1$. Therefore, $\frac{d_1}{r_1}-\frac{d_2}{r_2}=\frac{1}{d_1d_2}>0$. This proves that $\bm{\alpha}\in HN(\alpha)$.
\end{proof}

\begin{remark}
 When $X$ is a smooth projective curve of genus $g\geq 2$, we can prove a similar result which we mention here, but we will not need it in this paper.
 \end{remark}
 \begin{proposition}
 \label{arbitrarycurve}
  Let $X$ be a smooth projective curve of genus $g\geq 2$ and $\alpha\in\Z^+$. Then, the stack $\mathfrak{Coh}_{\alpha}(X)$ has a codimension one Harder-Narasimhan stratum if and only if $\alpha=(1,d)$. In this case, it is unique and corresponds to the Harder-Narasimhan type $((0,1),(1,d-1))$.
 \end{proposition}

\begin{proof}
 Write $\alpha=(r,d)$. Let $\bm{\alpha}=(\alpha_1,\hdots,\alpha_s)\in HN(\alpha)$. Consider the projection $p_{\bm{\alpha}} : \mathfrak{Coh}_{\bm{\alpha}}(X)\rightarrow \prod_{i=1}^s\mathfrak{Coh}_{(\alpha_i)}(X)$. It is a smooth morphism of relative dimension $-\sum_{j<i}\langle\alpha_i,\alpha_j\rangle$ where $\langle-,-\rangle$ is the Euler form (see Formula \eqref{eulerform}). Therefore,
 \[
  \dim\mathfrak{Coh}_{\bm{\alpha}}(X)=\sum_{i=1}^s(g-1)r_i^2+(g-1)\sum_{j<i}r_ir_j-\sum_{j<i}(r_id_j-r_jd_i).
 \]
Since $\dim\mathfrak{Coh}_{\alpha}(X)=(g-1)\left(\sum_{i=1}^sr_i\right)^2$, we have,
\[
 \dim\mathfrak{Coh}_{\bm{\alpha}}(X)-\dim\mathfrak{Coh}_{\alpha}(X)=-(g-1)\sum_{j<i}r_ir_j-\sum_{j<i}(r_id_j-r_jd_i).
\]
By the condition on the slopes, for any $j<i$, $r_id_j-r_jd_i>0$. If $s\geq 3$, then the codimension of $\mathfrak{Coh}_{\bm{\alpha}}(X)$ in $\mathfrak{Coh}_{\alpha}(X)$ is at least two (even three). If $\bm{\alpha}=(\alpha_1,\alpha_2)$,
\[
 \dim\mathfrak{Coh}_{\bm{\alpha}}(X)-\dim\mathfrak{Coh}_{\alpha}(X)=-(g-1)r_1r_2-(r_2d_1-r_1d_2)
\]
so if $r_1r_2\neq 0$, then the codimension is at least two. If $r_1r_2=0$, we have $r_1=0$ for slope reasons. Then, the codimension is one if and only if $r_2=1$ and $d_1=1$. So $\bm{\alpha}=((0,1),(1,d_2))$ This proves the proposition.
\end{proof}

\subsubsection{Local systems on the semistable locus}
\begin{proposition}
\label{extensionlocsys}
Let $X$ be a smooth projective curve, $\alpha\in\Z^+$ and let $\mathscr{L}$ be a local system on the semistable Harder-Narasimhan stratum $\mathfrak{Coh}_{(\alpha)}\subset \mathfrak{Coh}_{\alpha}$. Then, $\mathscr{L}$ extends to a local system on $\mathfrak{Coh}_{\alpha}(X)$.
\end{proposition}

\begin{proof}
 We assume $X$ is an elliptic curve. The same arguments combined with Proposition \ref{arbitrarycurve} give a proof for curves of genus $g\geq 2$. If $\gcd(\alpha)>1$, then the closed complement $\mathfrak{Coh}_{\alpha}(X)\setminus\mathfrak{Coh}_{(\alpha)}(X)$ of the open substack $\mathfrak{Coh}_{(\alpha)}(X)$ of $\mathfrak{Coh}_{\alpha}(X)$ is of codimension at least two by Proposition \ref{codimonestratum}. Since any local system extends over closed substacks of codimension at least two, $\mathscr{L}$ extends to a local system on $\mathfrak{Coh}_{\alpha}(X)$.
 
 If $\gcd(\alpha)=1$, we let $\bm{\alpha}=(\alpha_1,\alpha_2)$ be the Harder-Narasimhan type of the codimension one stratum of $\mathfrak{Coh}_{\alpha}(X)$ (see Proposition \ref{codimonestratum}). The open substack $\mathfrak{Coh}_{\alpha}(X)\cup\mathfrak{Coh}_{\bm{\alpha}}(X)$ is of codimension at least two. It suffices to show that $\mathscr{L}$ extends over this open substack. Since $\begin{pmatrix}
  r&d\\
  r_1&d_1
  \end{pmatrix}
$ has determinant one (see the proof of Proposition \ref{codimonestratum}), there exists $\gamma\in\SL_2(\Z)$ such that $\gamma\cdot\alpha=(1,1)$ and $\gamma\cdot \alpha_1=(0,1)$. By the isomorphism \eqref{isostra} of Section \ref{slaction}, it suffices to consider the case when $\alpha=(1,1)$. In this case, the result is implied by Proposition \ref{locsyspicard}
\end{proof}

\subsection{Twisted spherical Eisenstein perverse sheaves on the stack of coherent sheaves on an elliptic curve}

\subsubsection{The surface braid group}
Let $X$ be a connected topological surface. We will mainly be interested in the case when $X$ is an elliptic curve. Let $n\in\N$. The pure braid group $P_n(X)$ is by definition the fundamental group of $X^n\setminus \Delta$ while the braid group $B_n(X)$ is the fundamental group of $S^nX\setminus\Delta$. The $\mathfrak{S}_n$-covering $p_n : X^n\setminus\Delta\rightarrow S^nX\setminus\Delta$ induces an exact sequence of groups:
\[
 1\rightarrow P_n(X)\rightarrow B_n(X)\rightarrow\mathfrak{S}_n\rightarrow 1.
\]

\begin{proposition}
\label{canoque}
 There is a canonical quotient $B_n(X)\rightarrow K$ making the following diagram commute:
 \[
  \xymatrix{
   1\ar[r]&P_n(X)\ar[r]\ar@{->>}[d]& B_n(X)\ar[r]\ar@{->>}[d]&\mathfrak{S}_n\ar[r]\ar@{=}[d]& 1\\
   1\ar[r]& \pi_1(X)^n\ar[r]&K\ar[r]&\mathfrak{S}_n\ar[r]&1
  }
 \]

\end{proposition}
\begin{proof}
 The surjective morphism $P_n(X)\rightarrow \pi_1(X^n)\simeq \pi_1(X)^n$ is induced by the inclusion $X^n\setminus\Delta\rightarrow X^n$ and the group $K$ is constructed by push-out of the upper exact sequence.
\end{proof}
\begin{remark}
 It is possible to show that $K$ is isomorphic to the wreath product $\pi_1(X)^n\rtimes\mathfrak{S}_n$. Indeed, we can construct a section to the projection $K\rightarrow\mathfrak{S}_n$ as follows. We fix a set $Q=(q_1,\hdots,q_n)$ of $n$ distinct points on $X$ and interpret $B_n(X)$ as the group of braids on $X$, that is of isotopy classes of collections of $n$ paths $p=(p_1,\hdots,p_n)$ parametrized by $[0,1]$ on $X$ starting and ending at $Q$ such that for any $t\in[0,1]$ and any $i\neq j$, $p_i(t)\neq p_j(t)$. The left-most quotient is the quotient by the (normal) subgroup generated by braids whose strands are trivial in $\pi_1(X)$. For clarity, we assume that $Q$ is contained in an open subset $D$ of $X$ homeomorphic to a disk. For $\sigma\in\mathfrak{S}_n$, we choose a braid $p_{\sigma}\in B_n(X)$ such that each strand $p_{\sigma,i}$ is contained in $D$ (this is possible since $S^nD\setminus\Delta$ is path-connected) and such that $p_{\sigma}$ induces the permutation $\sigma$. The section of $K\rightarrow \mathfrak{S}_n$ is then the composition of the map $\mathfrak{S}_n\rightarrow \B_n(X)$, $\sigma\mapsto p_{\sigma}$, with the projection $B_n(X)\rightarrow K$. It is easily seen that two different choices for $p_{\sigma}$ give the same element in $K$ so that the composition is a group homomorphism.
\end{remark}

When $X$ is a projective algebraic curve over $\C$, we can consider all local systems, that is representations of $\pi_1(X)$ and we can also work with local systems coming from finite coverings of $X$, that is with representation of the \'etale fundamental group $\pi_1^{\acute{e}t}(X)$ which is the profinite completion $\widehat{\pi_1(X)}$ of $\pi_1(X)$. In the case when $X$ is an elliptic curve, $\pi^{\acute{e}t}_1(X)=\hat{\Z}\times\hat{\Z}$ and its representations correspond to local systems on $X$ having finite monodromy. Representations of $K$ are those local systems on $S^n(X)\setminus \Delta$ whose pullback by the $\mathfrak{S}_n$-covering $p_n : X^n\setminus\Delta\rightarrow S^nX\setminus\Delta$ extends to $X^n$.

The following remark will be useful in Section \ref{simpletwisted}. Let $\rho : P_n(X)\rightarrow\GL(V)$ be an irreducible representation of $P_n(X)$ and $\mathscr{L}$ be the corresponding local system on $X^n\setminus\Delta$. The local system $(p_n)_*\mathscr{L}$ on $S^nX\setminus\Delta$ is associated to the induced representation $\Ind_{P_n(X)}^{B_n(X)}(\rho)$. The pull-back $p_n^*(p_n)_*\mathscr{L}$ is the restriction to $P_n(X)$ of $\Ind_{P_n(X)}^{B_n(X)}(\rho)$. Its decomposition into irreducible representations of $P_n(X)$ is
\[
 \bigoplus_{\omega P_n(X)\in B_n(X)/P_n(X)=\mathfrak{S}_n} \omega\cdot\rho
\]
where $\omega\cdot\rho : P_n(X)=\pi_1(X^n\setminus\Delta)\rightarrow\GL(V)$, $g\mapsto \rho(\omega g\omega^{-1})$. This is the representation obtained from $\rho$ by permuting according to $\omega$ the factors of $X^n\setminus\Delta$. Therefore, when $\rho$ factors through $\pi_1(X)^n$ and hence corresponds to a local system of the form $\mathscr{L}=\mathscr{L}_1\boxtimes\hdots\hdots\boxtimes\mathscr{L}_n$ on $X^n$, $p_n^*(p_n)_*\mathscr{L}_{X^n\setminus\Delta}$ is a local system on $X^n\setminus\Delta$ which extends to $X^n$ and this extension is the local system
\begin{equation}
\label{directsum}
 \bigoplus_{\sigma\in\mathfrak{S}_n}(\mathscr{L}_{\sigma(1)}\boxtimes\hdots\boxtimes\mathscr{L}_{\sigma(n)}).
\end{equation}
Consequently, if $\mathscr{L}'$ is a simple direct summand of $(p_n)_*\mathscr{L}_{X^n\setminus\Delta}$, $p_n^*\mathscr{L}'$ is a local system on $X^n\setminus\Delta$ which extends to $X^n$ and the extension is a direct sum of some of the local systems appearing in the direct sum \eqref{directsum}. Consequently, if $\mathscr{L}_i, \mathscr{L}'_i$, $1\leq i\leq n$ are two collections of simple local systems on $X$ such that that the second is not a permutation of the first, then, letting $\mathscr{L}=\mathscr{L}_1\boxtimes\hdots\boxtimes\mathscr{L}_n$ and $\mathscr{L}'=\mathscr{L}'_1\boxtimes\hdots\boxtimes\mathscr{L}'_n$, the simple direct summands of $(p_n)_*\mathscr{L}$ and $(p_n)_*\mathscr{L}'$ are pairwise non-isomorphic.

\subsubsection{A class of perverse sheaves on the moduli stack of coherent sheaves on an elliptic curve}
In this Section, we define a category of perverse sheaves on the moduli stack of coherent sheaves on an elliptic curve. We will call these sheaves \emph{twisted spherical Eisenstein sheaves}. Let $\alpha\in\Z^+$. We let $\mathcal{P}^{\alpha}_{tw}$ be the semisimple subcategory of $\Perv(\mathfrak{Coh}_{\alpha}(X))$ whose simple objects are the simple perverse sheaves appearing with a possible shift as a direct summand of the induction of the perverse sheaves $\ICC(\mathscr{L}_i)$, $1\leq i\leq t$ for some $t\geq 1$, $\mathscr{L}_i$ local systems on $\mathfrak{Coh}_{(\alpha_i)}(X)$ with coprime $\alpha_i\in\Z^+$ such that $\sum_{i=1}^t\alpha_i=\alpha$. We denote $\mathcal{P}_{tw,f}^{\alpha}$ the category obtained in a similar way when allowing only local systems with finite monodromy. We will respectively denote $\mathcal{Q}_{tw}^{\alpha}$ and $\mathcal{Q}_{tw,f}^{\alpha}$ the full semisimple triangulated subcategories of $D(\mathfrak{Coh}_{\alpha}(X))$ generated by $\mathcal{P}^{\alpha}_{tw}$ (resp. $\mathcal{P}^{\alpha}_{tw,f}$). Since the map $q$ of the induction diagram (Section \ref{eisensteinsheaves}) is smooth and the map $p$ is proper, by the argument of \cite[\S 3.3]{MR2942792} involving the decomposition theorem, the induction functor induces a functor
\[
 \Ind_{\beta,\alpha}:\mathcal{Q}^{\beta}_{\sharp}\boxtimes\mathcal{Q}^{\alpha}_{\sharp}\rightarrow \mathcal{Q}^{\alpha+\beta}_{\sharp}
\]
for $\sharp=tw$ and $\sharp=tw,f$. For $\sharp=tw$, we need the generalization of the decomposition theorem which applies to any semisimple perverse sheaves and not only to whose of geometric origin (see for example \cite{MR3595140}). The next proposition shows that in the definition of $\mathcal{P}^{\alpha}_{\sharp}$ ($\sharp=tw$ or $\sharp=tw,f$), it suffices to consider inductions for $\rk(\alpha_i)\leq 1$.

\begin{lemma}
\label{linebun}
 Let $\alpha\in\Z^+$ and $\mathscr{L}$ be a local system on $\mathfrak{Coh}_{\alpha}(X)$. If $\rk(\alpha)>1$, there exists $d'\in\Z$ such that $\gcd(r-1,d-d')=1$, a local system $\mathscr{L}'$ on $\mathfrak{Coh}_{(1,d')}(X)$ and a local system $\mathscr{L}''$ on $\mathfrak{Coh}_{(r-1,d-d')}(X)$ such that $\mathscr{L}$ is a direct summand (possibly shifted) of the induction $\Ind_{(r-1,d-d'),(1,d')}(\mathscr{L}''\boxtimes\mathscr{L}')$.
 \end{lemma}
\begin{proof}
 Let $d'\in\Z$, $d'\ll 0$. Then, any nonzero morphism from a line bundle of degree $d'$ to a semistable coherent sheaf of class $\alpha$ is injective. We choose $d'$ so that $(r-1,d-d')$ is coprime. We let $\alpha'=(1,d)$ and $\alpha''=(r-1,d-d')$. The restriction diagram is
 \[
  \xymatrix{
  \mathfrak{Coh}_{\alpha''}(X)\times\mathfrak{Coh}_{\alpha'}(X)&\ar[l]_(.4)q\mathfrak{Exact}_{\alpha',\alpha''}\ar[r]^p&\mathfrak{Coh}_{\alpha}(X)
  }.
 \]
The map $p$ is proper, and surjective by our condition on $d'$. Therefore, $\mathscr{L}$ appears up to a shift in $p_*p^*\mathscr{L}$. Since $p^*\mathscr{L}$ is a local system on $\mathfrak{Exact}_{\alpha',\alpha''}$ and $q$ is a vector bundle stack, there exists a local system $\tilde{\mathscr{L}}$ on $\mathfrak{Coh}_{\alpha''}(X)\times\mathfrak{Coh}_{\alpha'}(X)$ such that $q^*\tilde{\mathscr{L}}=p^*\mathscr{L}$. We write $\tilde{\mathscr{L}}=\mathscr{L}''\boxtimes\mathscr{L}'$ for local systems $\mathscr{L}'$ (resp. $\mathscr{L}''$) on $\mathfrak{Coh}_{\alpha'}(X)$ (resp. $\mathfrak{Coh}_{\alpha''}(X)$). Then, $\mathscr{L}$ appears up to a shift in $\Ind_{\alpha'',\alpha'}(\mathscr{L}''\boxtimes\mathscr{L}')$.
\end{proof}

As in Section \ref{catsheaf}, for $\alpha\in\Z^+$ and $\bm{\alpha}\in HN(\alpha)$, we let $\mathcal{P}^{\bm{\alpha}}_{tw}$ (resp. $\mathcal{P}^{\bm{\alpha}}_{tw,f}$) denote the additive subcategory of perverse sheaves $(j_{(\alpha)}^*)\mathscr{F}=\mathscr{F}_{\mathfrak{Coh}_{\geq\bm{\alpha}}(X)}$ where $\mathscr{F}\in\mathcal{P}^{\alpha}_{tw}$ (resp. $\mathcal{P}^{\alpha}_{tw,f}$) satisfies $\supp\mathscr{F}=\overline{\supp\mathscr{F}\cap \mathfrak{Coh}_{\bm{\alpha}}(X)}$.

\subsection{Perverse sheaves with nilpotent singular support on the semistable locus}
\label{simpletwisted}
 In this section, we describe the simple perverse sheaf on the semistable locus of the stack of coherent sheaves on an elliptic curve whose singular support is nilpotent. We first define a family of local systems on the open substack of torsion sheaves of degree $n$ supported at $n$ pairwise distinct points of $X$. Consider the stack of torsion sheaves of degree $n$ on $X$, $\mathfrak{Tor}_n$. The open substack $\mathfrak{Tor}^{rss}_n$ of torsion sheaves of degree $n$ supported at $n$ distinct points is isomorphic to $(S^nX\setminus\Delta)/\G_m^n$ where the action of $(\G_m)^n$ is trivial. It therefore admits a $\mathfrak{S}_n$-cover $p_n : (X^n\setminus\Delta)/\G_m^n\rightarrow \mathfrak{Tor}_n^{rss}$. Let $\mathscr{L}_i$ be simple local systems on $X$ for $1\leq i\leq n$. When a basis of $\pi_1(X)\simeq \Z^2$ is fixed, this datum is equivalent to a $n$-uplet $\bm{z}\in((\C^*)^2)^n$ describing the monodromy. The exterior product $\mathscr{L}_1\boxtimes\hdots\boxtimes\mathscr{L}_n$ is a local system on $X^n$. We consider its restriction to $X^n\setminus\Delta$ and let $\mathscr{L}$ be the induced local system on $(X^n\setminus\Delta)/\G_m^n$. Then, the (underived) pushforward $(p_n)_*\mathscr{L}$ decomposes as a direct sum of local systems on $\mathfrak{Tor}_n^{rss}$ indexed by representations of the symmetric group:
\[
 \pi_{*}\mathscr{L}\simeq\bigoplus_{\lambda\in\mathscr{P}_n}\mathscr{L}_{\bm{z},\lambda}\otimes V_{\lambda}
\]
where $V_{\lambda}$ is the multiplicity vector space ($\mathfrak{S}_n$ acts on $\pi_*\mathscr{L}$ and this is its decomposition in isotypical components, where each $\mathscr{L}_{\bm{z},\lambda}$ is assumed to be irreducible). The multiplicity complexes $V_{\lambda}$ are nonzero. One can see this as follows. Since $\mathfrak{S}_n$ acts on $\pi_*\mathscr{L}$, we have an algebra morphism $\C[\mathfrak{S}_n]\rightarrow \End(\pi_*\mathscr{L})$. This algebra morphism is injective since the fiber of $\pi$ over any point $x\in S^nX\setminus\Delta$ has $n!$ points and $(\pi_*\mathscr{L})_x\simeq \C^{n!}$ is acted on by $\mathfrak{S}_n$ by the regular representation. Hence, the composition $\C[\mathfrak{S}_n]\rightarrow \End(\pi_*\mathscr{L})\rightarrow \End((\pi_*\mathscr{L})_x)$ is an isomorphism. In particular, this gives non-trivial orthogonal idempotents $1_{\lambda}$ in $\End(\pi_*\mathscr{L})$, one for each partition $\lambda$ of $n$. On the other hand, $\End(\pi_*\mathscr{L})=\bigoplus_{\lambda\in\mathscr{P}_n}\Hom_{\C}(V_{\lambda},V_{\lambda})$. The non-trivial idempotents of this algebra are the identity morphisms $\id_{V_{\lambda}}\in\Hom(V_{\lambda},V_{\lambda})$. Consequently, none of the multiplicity vector spaces $V_{\lambda}$ is trivial. Therefore, all the local systems $\mathscr{L}_{\bm{z},\lambda}$ occur in the direct sum decomposition. Moreover, by the discussion following Proposition \ref{canoque}, if $\bm{z}$ and $\bm{z}'$ cannot be obtained from each other by permuting the factors of $((\C^*)^2)^n$, then for any partitions $\lambda$ and $\nu$, $\mathscr{L}_{\bm{z},\lambda}$ and $\mathscr{L}_{\bm{z}',\nu}$ are not isomorphic. Of course, $\pi_*\mathscr{L}$ only depends on the local systems $\mathscr{L}_i$ up to permutation. 

Let $\alpha\in\Z^+$. By the isomorphism $\epsilon_{\alpha}:\mathfrak{Tor}_{\delta}=\mathfrak{Coh}_{(0,\delta)}\rightarrow \mathfrak{Coh}_{(\alpha)}(X)$ (Section \ref{sstable}), we can transport the local systems on $\mathfrak{Tor}_{\delta}^{rss}$ on local systems on an open substack of $\mathfrak{Coh}_{\alpha}(X)$. These are still denoted $\mathscr{L}_{\bm{z},\lambda}$ for $\bm{z}\in((\C^*)^2)^{\delta}$ and $\lambda\in\mathscr{P}_{\delta}$.

\begin{proposition}
\label{miclocss}
 Let $\alpha\in\Z^+$ and $\mathscr{F}$ be a simple perverse sheaf on $\mathfrak{Coh}_{(\alpha)}(X)$ having a nilpotent singular support. Then, $\mathscr{F}$ is the intersection cohomology sheaf $\ICC(\mathscr{L}_{\bm{z},\bm{\lambda}})$ of one of the local systems defined above.
\end{proposition}
\begin{proof}
 By the isomorphisms $\epsilon_{\alpha}$ between $\mathfrak{Tor}_{\delta}(X)$ and $\mathfrak{Coh}_{(\alpha)}(X)$, it suffices to consider the case when $\alpha=(0,d)$ for some $d\geq 1$. The singular support only depends on the local behaviour of the perverse sheaf, so we can assume that $X=\A^1$ is the affine line (as in the proof of \cite[Th\'eor\`eme (3.3.13)]{MR899400}). In this case, $\mathfrak{Coh}_{(0,d)}(X)\simeq \mathfrak{gl}_d/\GL_d$ and the nilpotent cone is $\mathscr{N}_{\gl_d}=\{(x,\xi)\in \mathfrak{gl}_d\times\mathscr{N}\mid [x,\xi]=0\}/\GL_d$. By Springer theory, perverse sheaves on $\mathfrak{gl}_d/\GL_d$ with singular support in $\mathscr{N}_{\mathfrak{gl}_d}$ are the Fourier-Sato transforms of the intersection cohomology sheaves of nilpotent orbits of $\mathfrak{gl}_d$. These are given by the local systems $\mathscr{L}_{\lambda}$ on the open substack $\mathfrak{gl}_d^{rss}/\GL_d\simeq (\A^n\setminus\Delta)/\mathfrak{S}_n$ of regular semisimple elements which appears in $(\pi_{\mathfrak{g}})_!\underline{\C}_{\tilde{\mathfrak{g}}^{rss}}$ (see Section \ref{singsuppsup}). This proves the proposition.
\end{proof}

\subsection{Proof of Theorem \ref{microlocchar}}
\label{psnilsing}

For the convenience of the reader, we recall Theorem \ref{microlocchar}.
\begin{theorem}
\label{miclocchar}
 The simple objects of the category $\Perv(\mathfrak{Coh}_{\alpha}(X),\mathscr{N}_{\alpha})$ of perverse sheaves having a nilpotent singular support are precisely the simple twisted spherical Eisenstein perverse sheaves.
\end{theorem}

\begin{proof}
 We first show that simple twisted spherical Eisenstein perverse sheaves have nilpotent singular support. This follows from the definition of these perverse sheaves as direct summands of induction of perverse sheaves of the form $\ICC(\mathscr{L})$ where $\mathscr{L}$ is a shifted local system on $\mathfrak{Coh}_{(\alpha)}(X)$, $\gcd(\alpha)=1$ by the same argument as in the proof of Proposition \ref{singsuppnil} and the fact that $\ICC(\mathscr{L})$ is by Proposition \ref{extensionlocsys} a (shifted) local system on $\mathfrak{Coh}_{\alpha}(X)$, so its singular support is the zero section of $T^*\mathfrak{Coh}_{\alpha}(X)$. 
 
 Conversely, we need to show that a simple perverse sheaf on $\mathfrak{Coh}_{\alpha}(X)$ having nilpotent singular support belongs to $\mathcal{P}^{\alpha}_{tw}$. Let $\mathscr{F}$ be such a perverse sheaf. Let $\mathfrak{Coh}_{\bm{\alpha},\bm{\lambda}}(X)$ be the stratum of $\mathfrak{Coh}_{\alpha}(X)$ such that $\supp\mathscr{F}=\overline{\mathfrak{Coh}_{\bm{\alpha},\bm{\lambda}}(X)}$ (which exists since $\supp\mathscr{F}=\pi_{\alpha}(SS(\mathscr{F}))$). We consider the iterated induction diagram restricted to the HN-stratum of type $\bm{\alpha}=(\alpha_s,\hdots,\alpha_1)$:
 \[
  \xymatrix{
  \prod_{i=1}^s\mathfrak{Coh}_{(\alpha_{i})}(X)&\ar[l]_(.35)q\mathfrak{Exact}_{\bm{\alpha}}\ar[r]^p&\mathfrak{Coh}_{\bm{\alpha}}(X)
  }
 \]
The map $p$ is an isomorphism and $q$ a vector bundle stack. Let $\mathscr{H}=p^*\mathscr{F}_{\mathfrak{Coh}_{\bm{\alpha}}(X)}$. Write $\bm{\lambda}=(\lambda_1,\hdots,\lambda_s)$. The singular support of the restriction of $\mathscr{H}$ to $q^{-1}\left(\prod_{i=1}^s\mathfrak{Coh}_{(\alpha_i),\lambda_i}\right)=p^{-1}(\mathfrak{Coh}_{\bm{\alpha},\bm{\lambda}}(X))$ is the conormal bundle to $p^{-1}(\mathfrak{Coh}_{\bm{\alpha},\bm{\lambda}}(X))$. Therefore, $\mathscr{H}=\ICC(\mathscr{L})$ for a local system $\mathscr{L}$ on $p^{-1}(\mathfrak{Coh}_{\bm{\alpha},\bm{\lambda}}(X))$. Hence, there exists a local system $\mathscr{L}'$ on $\prod_{i=1}^s\mathfrak{Coh}_{(\alpha_{i}),\lambda_i}(X)$ such that $\mathscr{L}=q^*\mathscr{L}'$. We can write $\mathscr{L}'=\mathscr{L}_1\boxtimes\hdots\boxtimes\mathscr{L}_s$ for local systems $\mathscr{L}_i$ on $\mathfrak{Coh}_{(\alpha),\lambda_i}$, $1\leq i\leq s$. Then, $\mathscr{H}=q^*(\mathscr{G}_1\boxtimes\hdots\boxtimes\mathscr{G}_s)[\dim q]$ and $\mathscr{F}$ is a simple constituent (up to some shift) of the induction $\Ind_{\alpha_1,\hdots,\alpha_s}(\mathscr{G}_1\boxtimes\hdots\boxtimes\mathscr{G}_s)$ for $\mathscr{G}_i=\ICC(\mathscr{L}_i)$. Since the singular support of a product is the product of the singular supports and $q$ is compatible with the stratifications (by definition), for any $1\leq i\leq s$, $\mathscr{G}_i$ is a perverse sheaf on $\mathfrak{Coh}_{(\alpha_i)}(X)$ with nilpotent singular support. By Proposition \ref{miclocss}, for $1\leq i\leq s$, $\mathscr{G}_i$ is a twisted spherical Eisenstein perverse sheaf, so the same is true for $\mathscr{F}$.
\end{proof}

\subsection{The simple twisted spherical Eisenstein perverse sheaves}

In this section, we describe explicitly the simple objects of the categories $\mathcal{P}^{\alpha}_{tw}$ and $\mathcal{P}^{\alpha}_{tw,f}$ in the spirit of \cite[Proposition 3.4]{MR2942792}. We let $\bm{\mu}\subset \C^*$ be the subgroup of roots of unity.

\begin{proposition}
 The simple objects of $\mathcal{P}^{(\alpha)}_{tw}$ (resp. $\mathcal{P}^{(\alpha)}_{tw,f}$) are the intermediate extensions $\ICC(\mathscr{L}_{\bm{z},\lambda})$ of the local systems described above for $\bm{z}\in((\C^*)^2)^{\delta}$ (resp. $\bm{z}\in(\bm{\mu}^2)^{\delta}$). Moreover, the local systems $\mathscr{L}_{\bm{z},\lambda}$ and $\mathscr{L}_{\bm{z}',\lambda'}$ on $\mathfrak{Coh}_{\alpha}(X)$, for some $\bm{z},\bm{z}'\in((\C^*)^2))^{\delta}$ are not isomorphic if $\bm{z}$ and $\bm{z}'$ cannot be deduced from each other by a permutation.
 
 Furthermore, for each $\alpha\in\Z^+$, there is a canonical bijection
 \[
  \theta_{\alpha} : \mathcal{P}^{\alpha}_{\sharp}\rightarrow \bigsqcup_{\bm{\alpha}=(\alpha_1,\hdots,\alpha_s)\in HN(\alpha)}\prod_{i=1}^s\mathcal{P}_{\sharp}^{(\alpha_i)}
 \]
($\sharp=tw$ or $\sharp=tw,f$) such that if $\theta(\mathscr{F})=(\mathscr{F}_1,\hdots,\mathscr{F}_s)$, then $\Ind_{\bm{\alpha}}(\mathscr{F}_s\boxtimes\hdots\boxtimes\mathscr{F}_1)=\mathscr{F}\oplus\mathscr{G}$ where $\supp\mathscr{G}\subsetneq\supp\mathscr{F}$. Last, $\D(\ICC(\mathscr{L}_{\bm{z},\lambda}))=\ICC(\mathscr{L}_{\bm{z}^{-1},\lambda})$ where $\D$ denotes the Verdier duality.
\end{proposition}

\begin{proof}
 The statement involving the Verdier duality follows from the fact that the dual of a local system on an elliptic curve is the local system with the inverse monodromy (equivalently, the dual of a representation $\rho : \Z^2\rightarrow \C^*$ is the representation $\rho^{-1} : \Z^2\rightarrow \C^*$, $z\mapsto \rho(z)^{-1}$). The existence of the bijection $\theta_{\alpha}$ is analogous to that of the similar map of \cite[Proposition 3.4]{MR2942792} and follows from the fact that the map $p$ in the iterated induction diagram restricted to sheaves of HN-type $\bm{\alpha}$ \eqref{iteratedind} is an isomorphism.
 
 For the first statement, the intermediate extensions $\ICC(\mathscr{L}_{\bm{z},\bm{\lambda}})$ for $\bm{z},\bm{\lambda}$ as in the theorem are simple objects of $\mathcal{P}_{tw}^{\geq(\alpha)}$ or $\mathcal{P}_{tw,f}^{\geq(\alpha)}$. Indeed, if $\alpha=\delta\alpha'$ with $\alpha'$ coprime and $\bm{z}=(z_1,\hdots,z_{\delta})$, this intersection cohomology sheaf appears as a simple constituent of the induction $\Ind_{\alpha',\hdots,\alpha'}(\ICC(\mathscr{L}_{z_1}\boxtimes\hdots\boxtimes\mathscr{L}_{z_{\delta}}))$. To conclude, we need to show that when performing inductions, no more simple perverse sheaves on $\mathfrak{Coh}_{\alpha}(X)$ whose support intersects the semistable stratum appear. This is a consequence of Theorem \ref{miclocchar} above and of its proof since all sheaves appearing in the inductions have nilpotent singular support, which is also proved in the proof of Theorem \ref{miclocchar}.
\end{proof}

\section*{Acknowledgements}
I warmly thank Olivier Schiffmann for asking me the question leading to this paper, for useful discussions concerning his work on the elliptic Hall algebra and for careful reading of a previous version of this article.

\bibliographystyle{alpha}

\begin{thebibliography}{GPHS14}

\bibitem[Ati57]{MR131423}
Michael~F. Atiyah.
\newblock Vector bundles over an elliptic curve.
\newblock {\em Proc. London Math. Soc. (3)}, 7:414--452, 1957.



\bibitem[BL94]{MR1299527}
Joseph Bernstein and Valery Lunts.
\newblock {\em Equivariant sheaves and functors}, volume 1578 of {\em Lecture
  Notes in Mathematics}.
\newblock Springer-Verlag, Berlin, 1994.



\bibitem[{Boz}17]{2017bozec}
Tristan {Bozec}.
\newblock {Irreducible components of the global nilpotent cone}.
\newblock {\em arXiv e-prints}, page arXiv:1712.07362, December 2017.

\bibitem[BS12]{MR2922373}
Igor Burban and Olivier Schiffmann.
\newblock On the {H}all algebra of an elliptic curve, {I}.
\newblock {\em Duke Math. J.}, 161(7):1171--1231, 2012.

\bibitem[CLL{\etalchar{+}}18]{MR3874690}
Sabin Cautis, Aaron~D. Lauda, Anthony~M. Licata, Peter Samuelson, and Joshua
  Sussan.
\newblock The elliptic {H}all algebra and the deformed {K}hovanov {H}eisenberg
  category.
\newblock {\em Selecta Math. (N.S.)}, 24(5):4041--4103, 2018.

\bibitem[dC16]{MR3595140}
Mark Andrea~A. de~Cataldo.
\newblock Decomposition theorem for semi-simples.
\newblock {\em J. Singul.}, 14:194--197, 2016.

\bibitem[EM99]{MR1682280}
Sam Evens and Ivan Mirkovi\'{c}.
\newblock Characteristic cycles for the loop {G}rassmannian and nilpotent
  orbits.
\newblock {\em Duke Math. J.}, 97(1):109--126, 1999.

\bibitem[Gin01]{MR1853354}
Victor Ginzburg.
\newblock The global nilpotent variety is {L}agrangian.
\newblock {\em Duke Math. J.}, 109(3):511--519, 2001.

\bibitem[GPHS14]{MR3293805}
Oscar Garc\'{\i}a-Prada, Jochen Heinloth, and Alexander Schmitt.
\newblock On the motives of moduli of chains and {H}iggs bundles.
\newblock {\em J. Eur. Math. Soc. (JEMS)}, 16(12):2617--2668, 2014.

\bibitem[Gun18]{MR3874694}
Sam Gunningham.
\newblock Generalized {S}pringer theory for {$D$}-modules on a reductive {L}ie
  algebra.
\newblock {\em Selecta Math. (N.S.)}, 24(5):4223--4277, 2018.

\bibitem[Hei04]{MR2139694}
Jochen Heinloth.
\newblock Coherent sheaves with parabolic structure and construction of {H}ecke
  eigensheaves for some ramified local systems.
\newblock {\em Ann. Inst. Fourier (Grenoble)}, 54(7):2235--2325 (2005), 2004.

\bibitem[{Hen}20]{2020henn}
Lucien {Hennecart}.
\newblock {Microlocal characterization of Lusztig sheaves for affine quivers
  and $g$-loops quivers}.
\newblock {\em arXiv e-prints}, page arXiv:2006.12780, June 2020.

\bibitem[KS90]{MR1074006}
Masaki Kashiwara and Pierre Schapira.
\newblock {\em Sheaves on manifolds}, volume 292 of {\em Grundlehren der
  Mathematischen Wissenschaften [Fundamental Principles of Mathematical
  Sciences]}.
\newblock Springer-Verlag, Berlin, 1990.
\newblock With a chapter in French by Christian Houzel.

\bibitem[KS97]{MR1458969}
Masaki Kashiwara and Yoshihisa Saito.
\newblock Geometric construction of crystal bases.
\newblock {\em Duke Math. J.}, 89(1):9--36, 1997.

\bibitem[Kul90]{MR1074778}
Sergej~A. Kuleshov.
\newblock Construction of bundles on an elliptic curve.
\newblock In {\em Helices and vector bundles}, volume 148 of {\em London Math.
  Soc. Lecture Note Ser.}, pages 7--22. Cambridge Univ. Press, Cambridge, 1990.

\bibitem[Lau87]{MR899400}
G\'{e}rard Laumon.
\newblock Correspondance de {L}anglands g\'{e}om\'{e}trique pour les corps de
  fonctions.
\newblock {\em Duke Math. J.}, 54(2):309--359, 1987.

\bibitem[Lau88]{MR962524}
G\'{e}rard Laumon.
\newblock Un analogue global du c\^{o}ne nilpotent.
\newblock {\em Duke Math. J.}, 57(2):647--671, 1988.

\bibitem[Lau90]{MR1044822}
G.~Laumon.
\newblock Faisceaux automorphes li\'{e}s aux s\'{e}ries d'{E}isenstein.
\newblock In {\em Automorphic forms, {S}himura varieties, and {$L$}-functions,
  {V}ol. {I} ({A}nn {A}rbor, {MI}, 1988)}, volume~10 of {\em Perspect. Math.},
  pages 227--281. Academic Press, Boston, MA, 1990.

\bibitem[Lus91]{MR1088333}
George Lusztig.
\newblock Quivers, perverse sheaves, and quantized enveloping algebras.
\newblock {\em J. Amer. Math. Soc.}, 4(2):365--421, 1991.

\bibitem[Lus00]{MR1758244}
George Lusztig.
\newblock Semicanonical bases arising from enveloping algebras.
\newblock {\em Adv. Math.}, 151(2):129--139, 2000.

\bibitem[Mac15]{MR3443860}
Ian~G. Macdonald.
\newblock {\em Symmetric functions and {H}all polynomials}.
\newblock Oxford Classic Texts in the Physical Sciences. The Clarendon Press,
  Oxford University Press, New York, second edition, 2015.
\newblock With contribution by A. V. Zelevinsky and a foreword by Richard
  Stanley, Reprint of the 2008 paperback edition [ MR1354144].

\bibitem[Mir04]{MR2124171}
Ivan~Mirkovi\'{c}.
\newblock Character sheaves on reductive {L}ie algebras.
\newblock {\em Mosc. Math. J.}, 4(4):897--910, 981, 2004.

\bibitem[MS17]{MR3626565}
Hugh Morton and Peter Samuelson.
\newblock The {HOMFLYPT} skein algebra of the torus and the elliptic {H}all
  algebra.
\newblock {\em Duke Math. J.}, 166(5):801--854, 2017.

\bibitem[Rin98a]{MR1648647}
Claus~Michael Ringel.
\newblock The preprojective algebra of a quiver.
\newblock In {\em Algebras and modules, {II} ({G}eiranger, 1996)}, volume~24 of
  {\em CMS Conf. Proc.}, pages 467--480. Amer. Math. Soc., Providence, RI,
  1998.

\bibitem[Rin98b]{MR1676227}
Claus~Michael Ringel.
\newblock The preprojective algebra of a tame quiver: the irreducible
  components of the module varieties.
\newblock In {\em Trends in the representation theory of finite-dimensional
  algebras ({S}eattle, {WA}, 1997)}, volume 229 of {\em Contemp. Math.}, pages
  293--306. Amer. Math. Soc., Providence, RI, 1998.

\bibitem[Sch12]{MR2942792}
Olivier Schiffmann.
\newblock On the {H}all algebra of an elliptic curve, {II}.
\newblock {\em Duke Math. J.}, 161(9):1711--1750, 2012.

\bibitem[SS18]{2018arXivss}
Francesco {Sala} and Olivier {Schiffmann}.
\newblock {Cohomological Hall algebra of Higgs sheaves on a curve}.
\newblock {\em arXiv e-prints}, page arXiv:1801.03482, January 2018.

\bibitem[SV13]{MR3018956}
Olivier Schiffmann and Eric Vasserot.
\newblock The elliptic {H}all algebra and the {$K$}-theory of the {H}ilbert
  scheme of {$\Bbb A^2$}.
\newblock {\em Duke Math. J.}, 162(2):279--366, 2013.

\end{thebibliography}
\newcommand{\etalchar}[1]{$^{#1}$}

\end{document}